\documentclass[12pt, final]{article}

\RequirePackage{amsfonts}
\RequirePackage{amsmath}
\RequirePackage{alltt, moreverb}
\RequirePackage{showkeys, showidx}
\RequirePackage{graphicx, graphics}
\RequirePackage[utf8]{inputenc}
\RequirePackage{algorithm, algorithmic} 

\RequirePackage{color, colortbl}
\RequirePackage{rotating}
\RequirePackage{boxedminipage}
\RequirePackage{upgreek}
\RequirePackage{subfigure}

\usepackage{textcomp}

\RequirePackage{enumitem}
 
\usepackage{amsopn}
\RequirePackage{amsthm}
\RequirePackage{amsbsy}
\theoremstyle{plain} 
\newtheorem{theorem}{Theorem}[section]

\newtheorem{lemma}      [theorem]{Lemma}

\newcommand{\eps} {\varepsilon}                            
\DeclareMathOperator*{\argmin}{arg\,min}                   
\renewcommand{\t} {^{\top}}                                
\newcommand{\norm} [2][]{\left\|#2\right\|_{#1}}           

\newcommand{\bmx}{\left[ \begin{array}}
\newcommand{\emx}{\end{array} \right]}

\newcommand{\bfgamma}{{\boldsymbol{\gamma}}}

\newcommand{\bfA}{{\bf A}}

\newcommand{\bfJ}{{\bf J}}
\newcommand{\bfK}{{\bf K}}

\newcommand{\bfd}{{\bf d}}
\newcommand{\bfe}{{\bf e}}
\newcommand{\bff}{{\bf f}}
\newcommand{\bfg}{{\bf g}}

\newcommand{\bfr}{{\bf r}}
\newcommand{\bfs}{{\bf s}}

\newcommand{\bfzero}{{\bf0}}

\newcommand{\bbR}{\mathbb{R}}

\newcommand{\true}{\mathrm{true}}

\newcommand{\wf}{\mbox{WF}}
\newcommand{\graph}{\mbox{graph}}
\newcommand{\mT}{\mathcal{T}}
\newcommand{\mS}{\mathcal{S}}
\newcommand{\mE}{\mathcal{E}}
\newcommand{\mI}{\mathbb{I}}
\newcommand{\supp}{{\rm supp}}

\usepackage{tikz,pgflibraryplotmarks}
\usepackage{pgf,pgfplots,pgfarrows} 

\newdimen\iwidth
\newdimen\iheight
\usepackage{multirow}

\usepackage{amsmath,amssymb} 
\usepackage{algorithm,algorithmic}
\usepackage{color}

\definecolor{blue}{rgb}{0,0,1}
\definecolor{red}{rgb}{1,0,0}
\definecolor{green}{rgb}{0,1,0}

\newcommand{\bd}{{\bf d}}

\usepackage{fullpage}
\usepackage{marvosym}

\newcommand{\TheTitle}{Motion Estimation and Correction in Photoacoustic Tomographic Reconstruction}

\title{{\TheTitle}}

\author{
  Julianne Chung
	\thanks{Department of Mathematics, Virginia Tech, Blacksburg, VA \newline \hspace*{10ex}
    \Letter \ \texttt{jmchung@vt.edu} \ \ \Mundus \ {www.math.vt.edu/people/jmchung/}}
  \and
Linh Nguyen
	\thanks{Department of Mathematics, University of Idaho, Moscow, ID 83844 \newline \hspace*{10ex}
    \Letter \ \texttt{{lnguyen@uidaho.edu} \ \ \Mundus \ {http://webpages.uidaho.edu/lnguyen/}}}
}

\begin{document}

\maketitle

\begin{abstract}
Motion, e.g., due to patient movement or improper device calibration, is inevitable in many imaging modalities such as photoacoustic tomography (PAT) by a rotating system and can lead to undesirable motion artifacts in image reconstructions, if ignored. In this paper, we establish a hybrid-type model for PAT that incorporates motion in the model. We first introduce an approximate continuous model and establish two uniqueness results for simple parameterized motion models. Then we formulate the discrete problem of simultaneous motion estimation and image reconstruction as a separable nonlinear least squares problem and describe an automatic approach to detect and eliminate motion artifacts during the reconstruction process. Numerical examples validate our methods.
\end{abstract}

	
	\noindent{\it AMS}:
	 	53C65, 65F22, 92C55, 65R10, 65R32 
	
\thispagestyle{plain}

\section{Introduction} 
\label{sec:introduction}
Photoacoustic tomography (PAT) is a hybrid imaging technique that combines the high contrast of optical imaging with the high resolution of ultrasound imaging. A short pulse of laser light is irradiated through the biological object of interest. Due to the photoelastic effect, the object slightly expands and releases an ultrasound pressure that propagates through the space. The pressure is measured by transducers located on an observation surface $S$. The initial pressure $f(x)$ contains information regarding the inner structures of the object. One concentrates on finding $f(x)$ from the observed data.  

In this article, we are interested in a recent PAT implementation, where there are multiple illuminations and an array of transducers is rotated around the object of interest to collect the data. We will discuss that setup shortly. Meanwhile, let us introduce the standard setup of PAT (e.g., \cite{kruger,oraevsky03opto,wang13biomedical}), where the object is irradiated once and the data is simultaneously measured by all transducers covering the observation surface $S$. The mathematical model of PAT is given then by the wave equation (e.g., \cite{diebold1991photo,Tam})
\begin{eqnarray}
\left\{\begin{array}{l} u_{tt}(x,t) - c^2(x) \, \Delta u(x,t) = 0,\quad (x,t) \in \mathbb{R}^3 \times \mathbb{R}_+,\\[4 pt]
u(x,0) = f(x), \quad u_t(x,0)=0, \quad x \in \mathbb{R}^3.
\end{array} \right.
\end{eqnarray}
Here, function $c(x)$ represents the ultrasound speed, and the measured data is  $g=u|_{S \times \mathbb{R}_+}$. One needs to solve the \emph{inverse problem}: find $f$ given $g$.
This problem has been intensively studied from both theoretical and numerical perspectives (see, e.g., \cite{xu2002exact,XW05,xu2006photoacoustic,anastasio2005half,haltmeier2005filtered,FPR,Kun07,anastasio2007application,FHR,Kun07s,kuchment2008mathematics,haltmeier2009reconstruction,scherzer2009variational,HKN,kuchment2011mathematics,US,qian2011new,stefanov2011thermo,kuchment2014radon,kuchment2015mathematical}, just to name a few). Common techniques to solve the aforementioned inversion problem include explicit inversion formula, time reversal, and series solution.  Let us also mention that PAT in two dimensional space has also been well studied in the literature (see, e.g., \cite{Kun07,FHR}).

In many applications, it is reasonable to assume that $c(x)$ is constant, in which case the data can be represented by the spherical Radon transform $\mathcal{R}(f)$ of $f$ (see, e.g., \cite{FPR}):  
$$\mathcal{R}(f)(z,r) = \int_{\mathbb{S}(z,r)} f(x) \, d\sigma(x), \quad (z,r) \in S \times \mathbb{R}_+,$$ where $\mathbb{S}(z,r)$ is the spherical of radius $r$ centered at $z$.
Hence, the PAT reconstruction process
reduces to the inversion of the spherical Radon transform. When $S$ has special geometry such as a plane or ellipsoid, the spherical Radon transform admits an explicit inversion formula (e.g., \cite{FPR,Kun07,LinhNguyen2009,palamodov2012uniform,natterer2012photo,haltmeier2013inversion,haltmeier2014universal}). The corresponding problem in the two dimensional space reduces to the inversion of the circular Radon transform and has also been thoroughly studied (see, e.g., \cite{Kun07,FHR}). 




Various technical constraints in the standard PAT setup have motivated the development of newer PAT technologies, where transducers are rotated around the object and data is acquired in time. In this article, we are interested in such a 
PAT setup,  which is currently being developed at TomoWave Laboratories (see the company's webpage 
{http://www.tomowave.com/}
 and the reference \cite{brecht09whole}). In this setup, transducers are attached to a half-circle ring, placed vertically, and rotated around the object, see Figure~\ref{fig:problem}. After each illumination and measurement, the ring is slightly rotated. By rotating the ring, one can specify a sphere where the pressures are measured. Under ideal conditions, the full set of measurements is equivalent to that obtained using the standard setup with a spherical observation surface. 

\begin{figure}\label{fig:problem}
\begin{center}
	\includegraphics[width = 2.5in]{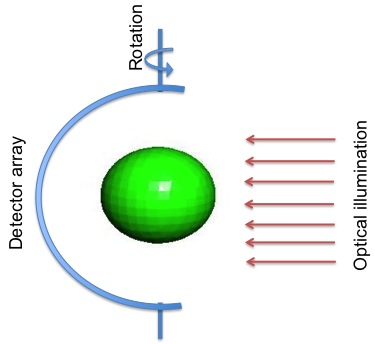}
	\end{center}
\caption{Setup of new time-lapse PAT device, in development at TOMOWAVE.}
\end{figure}

However, in practice, the object that is being imaged may move during the measurement process, thereby presenting a significant challenge to achieving accurate tomographic reconstruction for dynamic PAT systems. For example, in in-vivo imaging, the observations may be affected by the body's natural activities. If ignored, the resulting motion artifacts can inhibit reconstruction quality, thereby leading to incorrect image analysis and false diagnosis.  Although generic pre- or post-processing software may improve the visual appeal of reconstructions, in this paper, we consider a mathematical model that directly incorporates motion for more precise reconstruction. 
In X-ray CT, there has been work on estimating motion from the data (e.g., \cite{katsevich2011local}) and on incorporating motion in reconstruction algorithms (e.g., \cite{rit2009comparison,isola2008motion,van2007motion,blondel20043d,pack2004dynamic,grangeat2002theoretical,desbat2007compensation,taguchi2008motion,taguchi2007motion,katsevich2008motion,katsevich2010accurate,hahn2014efficient,hahn2014reconstruction,hahn2015dynamic,katsevich2016motion,hahn2016detectable}). However, to the best of our knowledge, there is currently no such work for PAT reconstruction, where the forward model is different than that of X-ray CT.  Our work is the first attempt to fill this gap. 

\paragraph{Overview of contributions}  This work has two main goals:
i) to establish a continuous approximate model of PAT reconstruction under motion and to derive two uniqueness results for this theoretical model, and ii) to propose an automatic and robust reconstruction technique for simultaneous estimation of the motion parameters and the desired image.
For the first goal, we will recast the imaging procedure as a hybrid-type inverse problem and approximate it by a continuous model. We will make use of the theory of analytic wave front set and convex analysis to establish two uniqueness results.  Then, for the second goal, since current techniques for standard PAT, such as inversion formula, time reversal, and series solution, do not seem feasible in our framework, our approach is to discretize the problem and employ algebraic techniques.  We reformulate the problem of simultaneous estimation of motion parameters and reconstruction as a separable nonlinear least squares problem and exploit recent work on variable projection approaches \cite{GoPe03,Chung2010b}, where regularization parameter selection can be done automatically. To handle large-scale problems and to avoid expensive computations, we make use of object-oriented programming and exploit matrix sparsity.




The article is organized as follows.  In Section~\ref{sec:model}, we describe the mathematical model for PAT reconstruction under motion.  Then, assuming that the object does not move so rapidly, we derive an approximate continuous mathematical model. In Section~\ref{S:unique}, we provide theoretical results and properties that help in understanding the underlying problem.  In Section~\ref{sec:methods} we describe an efficient computational approach to solve this problem.  Numerical results are provided in Section~\ref{sec:numerical_results}, and conclusions can be found in Section~\ref{sec:conclusions}.

\section{The mathematical model} \label{sec:model}

Let us begin by describing a mathematical model for PAT with motion. For simplicity, in this article we consider the two-dimensional model. Our ideas are extendable to the three dimensional problem; however, the computation and analysis in that case are more involved and will be addressed in a followup paper. For the two dimensional problem, we consider the case of one detector rotating on the unit circle $\mathbb{S}$ to measure the signal. Let us assume that the object is supported inside the unit disc $\mathbb{D}$ of $\mathbb{R}^2$. Let $\{z_i = (\cos \phi_i, \sin \phi_i)\}$, $i=1,\dots,n$, be $n$ locations distributed uniformly counterclockwise along the unit circle $\mathbb{S}$. Fixing $\Delta t >0$, for each moment $t_i = (i-1) \Delta t$, $i=1,\dots,n$, one illuminates the object uniformly and measures the ultrasonic pressure at $z_i$. Assuming the speed of sound is constant, the data can be considered as the circular Radon transform $\mathcal{R}(f_i)(\phi_i,r)$ of $f_i$:
\begin{equation*}\mathcal{R}(f_i)(\phi_i,r) = \int_{\mathbb{C}(\phi_i ,r)} f_i(x) \, ds(x),~ r \in \mathbb{R}_+ \quad i=1,\dots,n,
	\end{equation*} where $\mathbb{C}(\phi,r)$ denotes the circle of radius $r$ centered at $z=(\cos \phi, \sin \phi)$.
Here, $f_i$ is the object at moment $t_i$. It is commonly assumed in imaging science that $f_i$ is a deformation of the steady state $f$, which represents the object at rest. Our problem now reads as follows:

\begin{center}

\emph{find the function $f$ from $\{\mathcal{R}(f_i)(\phi_i,r): i = 1,\dots,n,~0 \leq r \leq 2\}.$}

\end{center}
The data of this problem is of hybrid-type: it is neither discrete (due to the continuous variable $r$) nor continuous (due to the discrete variable $\phi_i$).  Since the desired function $f(x)$ is of two continuous variables, the above problem is underdetermined and the kernel of the forward operator is nontrivial. Another underlying challenge of the above reconstruction problem is ill-posedness, since the forward operator is a smoothing operator. Moreover, the exact motion may not be known a priori. Numerical techniques to handle these challenges will be discussed in Section~\ref{sec:methods}.

\subsection{An approximate continuous model} \label{sec:continuous}
Assume that the object does not change rapidly and that the measurement time $\Delta t$ at each location is small. That is, the object does not change too much between two consecutive illuminations and measurements. Moreover, we assume that the number of measurement locations is big enough. Then, the data can be approximately considered as a function defined on $(\phi,r) \in (\alpha,\beta) \times \mathbb{R}_+$. For each location $z= (\cos \phi, \sin \phi)$, the data is $\{\mathcal{R}(f_\phi)(\phi,r): r>0\}$, where $f_\phi$ is the deformation of $f$ at time $t_\phi$ (when the signal is measured at location $z(\phi)$). Thus, the continuous problem, incorporating time-varying deformations, can be stated as:


\begin{center} 
(P) \quad	\emph{find $f$ given $\mathcal{R}(f_\phi)(\phi,r)$ for all $(\phi,r) \in (\alpha,\beta)  \times \mathbb{R}_+$.}
	\end{center} 
Here, the length of $(\alpha,\beta)$ may be less than $2 \pi$, when the rotation of the transducer does not make a complete circle. In that case, we assume that $(\alpha, \beta) \subset [- \pi/2,3 \pi/2]$. On the other hand, the length of $(\alpha,\beta)$ can also be bigger than $2 \pi$, when the transducer makes more than a complete circle. However, we note that, in the latter case, $f_\phi$ and $f_{\phi+2 \pi}$ may not be the same, if the motion is not periodic. 

The deformation relationship between $f$ and $f_\phi$ is a topic that needs investigation. In this article, we follow the literature in dynamic tomography (e.g., \cite{hahn2014reconstruction,hahn2015dynamic}) and assume that  \footnote{Other works assume a slightly different relation, which accounts for conservation of mass (e.g., \cite{katsevich2016motion})}.
$$f_\phi(\Phi(\phi,x)) = f(x).$$ Here for a fixed $\phi$, $\Phi(\phi,x)$ is the location at moment $t_\phi$ (when the signal is measured at location $z(\phi)=(\cos \phi, \sin \phi)$) of the particle whose location is $x$ when the object is at rest. We will make the following assumptions on the deformation mapping $\Phi$:
\begin{enumerate}
\item[(A.1)]$\Phi(\phi.,)$ is a bijection from $\mathbb{R}^2$ to $\mathbb{R}^2$. Physically, this means two points can not deform into one location. Then, $\Phi(\phi,.)$ has an inverse $\Psi(\phi,.)$ and
$$f_\phi(y) = f(\Psi(\phi,y)).$$
\item[(A.2)] There is  an open subset $\mathcal{O}$ and a compact subset $K$ of the unit open disc $\mathbb{D}$ such that $\supp(f) \subset \mathcal{O}$ and $\Phi(\phi,\mathcal{O}) \subset K$ for all $\phi \in (\alpha, \beta)$. That is, the the object is always contained inside $K$ during the imaging process.  
\item[(A.3)] $\Phi$ is an analytic function in both variables. This technical restriction allows us to obtain theoretical results.
\end{enumerate}
{\bf In general, the deformation mapping $\Phi$ may not be given.} We, instead, have to extract it from the measured data as well. This, in fact, poses the main challenges for PAT with motions, both theoretically and computationally. In this article, we analyze Problem~(P) under this circumstance. 

Next, let us write down the data $\mathcal{R}(f_\phi)$ in terms of $f$. We first notice that
\begin{eqnarray*} \mathcal{R}(f_\phi)(\phi,r) &=& \int_{\mathbb{R}^2} \delta(|z(\phi) - y|-r ) f_\phi(y) \, dy
\\ &=& \int_{\mathbb{R}} \int_{\mathbb{R}^2} e^{i \lambda (|z(\phi) - y|-r )} f_\phi(y) \, dy \, d \lambda.
\end{eqnarray*}
That is,
\begin{eqnarray*} \mathcal{R}(f_\phi)(\phi,r) &=& \int_{\mathbb{R}} \int_{\mathbb{R}^2} e^{i \lambda (|z(\phi) - y|-r )} f(\Psi(\phi,y)) \, dy \, d \lambda. 
\end{eqnarray*}
By the change of variables $x=\Psi(\phi,y)$, i.e., $y=\Phi(\phi,x)$, we obtain
\begin{eqnarray}\label{E:data} \mathcal{R}(f_\phi)(\phi,r)  &=& \int_{\mathbb{R}} \int_{\mathbb{R}^2} e^{i \lambda (|z(\phi) - \Phi(\phi,x)|-r )} J_x[\Phi(\phi,x)] \, f(x) \, dx \, d \lambda,
\end{eqnarray} where 
$$J_x[\Phi(\phi,x)] = |\det D_x \Phi(\phi,x)|$$
is the absolute value of the Jacobian of the change of variables $x \to y= \Phi(\phi,x)$. 
That is, 
\begin{eqnarray}\label{E:geo}  \nonumber \mathcal{R}(f_\phi)(\phi,r) &=& \int_{\mathbb{R}^2} \delta(|z(\phi) - \Phi(\phi,x)|-r ) \, J_x[\Phi(\phi,x)] \, f(x) \, dx  \\ \nonumber &=&  \int_{\mathcal{E}(\phi,r)} \frac{J_x[\Phi(\phi,x)] }{|\nabla_x |z(\phi) -\Phi(\phi,x)||}f(x) \, d\sigma(x) \\ &=&  \int_{\mathcal{E}(\phi,r)} \frac{J_x[\Phi(\phi,x)] \, |z(\phi) -\Phi(\phi,x)|}{|\left< z(\phi) - \Phi(\phi,x), D_x\Phi(\phi,x) \right> |}f(x) \, d\sigma(x).
\end{eqnarray}
Here,
\begin{eqnarray*} \mathcal{E}(\phi,r) &=& \{x \in \mathbb{R}^2: |z(\phi) - \Phi(\phi,x) |=r \}\\ &=&  \Phi^{-1} \{\mathbb{C}(\phi,r)\}  \\ &=& \Psi  \{\mathbb{C}(\phi,r)\}.
\end{eqnarray*}
Since $\Psi$ is analytic, $\mathcal{E}(\phi,r)$ is a closed analytic curve in $\mathbb{R}^2$. 

Let us now denote
$$ \mathcal{R}_\Phi(f)(\phi,r)= \int_{\mathcal{E}(\phi,r)} w(\phi,x) \, f(x)  \,d \sigma(x),$$ where $$w(\phi,x)=\frac{J_x[\Phi(\phi,x)]  \, |z(\phi) -\Phi(\phi,x)|}{|\left< z(\phi) - \Phi(\phi,x), \nabla_x \Phi(\phi,x)| \right> |}. $$
Now, Problem~(P) can be reformulated as:
\begin{center} 
	\emph{find $f$ given $\mathcal{R}_\Phi (f)(\phi,r)$ for all $(\phi,r) \in (\alpha,\beta) \times \mathbb{R}_+$.}
	\end{center} 

\medskip

Let us note that $\mathcal{R}_\phi$ is a generalized Radon transform (see, e.g., \cite{gel1969differential,quinto1980dependence,guillemin1985some,guillemin1990geometric,quinto1994radon,palamodov2012uniform}) with the incidence relation defined by
$$I(x,\phi,r):= \frac{1}{2} \big(|z(\phi) - \Phi(\phi,x)|^2 -r^2 \big) =0. $$
It is a Fourier integral operator with the canonical relation
$$\mathcal{C} = \{(\phi,r, \lambda \, \partial_\phi I, \lambda\, \partial_r I; x,\lambda \, \partial_x I): I(x,\phi,r) = 0, \lambda \neq 0\}.$$
It can be parametrized by 
$$(\phi,x,\lambda) \in  (\alpha,\beta) \times \mathcal{O} \times (\mathbb{R} \setminus 0).$$
Let us consider its left projection 
$$\pi_L: (\phi,x,\lambda)  \longmapsto \big(\phi, |z(\phi) - \Phi(\phi,x)| ,  \lambda \, \partial_\theta I, \lambda \, \partial_r I \big),$$
and right projection
$$\pi_R: (\phi,x,\lambda) \longmapsto  \big(x, \lambda \, \partial_x I \big).$$

In order to analyze Problem (P), we will assume the following condition: 
\begin{itemize}
\item[(A.5)]  The left projection $\pi_L$ is an injective immersion. This is the so-called Bolker condition (see, e.g., \cite{quinto1980dependence,guillemin1985some,guillemin1990geometric,quinto1994radon}). 
\item[(A.6)] The right projection $\pi_R$ is a surjective submersion. Intuitively, this condition means that all the singularities of $f$ are visible to the transform $\mathcal{R}_\Phi$.
\end{itemize}
As proved in Appendix~\ref{A:Bolker}, these two conditions are equivalent to
\begin{itemize}
\item[(B)] (Bolker's condition) For each $\phi \in (\alpha,\beta)$, the mapping $$\Pi: x \in \mathcal{O} \to (\frac{1}{2} |z(\phi) -\Phi(\phi,x)|^2 ,\frac{1}{2} \partial_\phi |z(\phi) -\Phi(\phi,x)|^2)$$ is an injective immersion.
\item[(C)] (Visibility condition) For each $x \in \mathcal{O}$, the set of unit normal vectors of all curves $\mathcal{E}(\phi,r)$ passing through $x$ covers the unit circle $\mathbb{S}$. 
\end{itemize}

Let us note that condition (B) was proposed in \cite{HomanZhou}. However, it was not proved to be equivalent to (A.5), since it was irrelevant for its purpose. This fact will be proved in Appendix~\ref{A:Bolker} for the sake of completeness. Although it is not trivial to verify conditions (B) and (C) for a general family of deformations $\Phi$, it is possible for simple cases. In particular, in Section~\ref{Ss:stretching}, we will present a model of deformation for which these conditions hold.

\subsection{A simple model: vertical stretching} \label{Ss:stretching} Consider $\Phi$ to be a family of vertical stretching in $\mathbb{R}^2$. That is,
$$\Phi(\phi,x) = (x_1, c+ a(\phi) (x_2-c) ),$$ where $a(\phi)>0$ is the stretch factor corresponding to the measurement at transducer $z(\phi)$. The horizontal line $x_2 = c$ is unchanged under the above vertical stretch. We will refer to it as the base line of the stretch. Let us prove that under appropriate assumptions, this model satisfies conditions (B) and (C).

\begin{lemma} \label{L:Cond2} Let $\eps>0$ be such that $K \subset \mathbb{D}_{1-\eps}$ and
$$C_\epsilon= \sup_{z \in \mathbb{S},~ x \in \mathbb{D}_{1-\eps}} \frac{|z -x|}{|\left<z -x,z \right>|} = \frac{1}{\sqrt{2 \eps - \eps^2}}.$$
Assume that $$\left|\frac{a'(\phi)}{a(\phi)} \right| \leq \frac{1}{(3+|c|) \, C_\epsilon} \mbox{ for all } z \in \mathbb{S}.$$ Then, the Bolker condition (B) holds. 
\end{lemma}

\begin{proof}   Let us first verify that the mapping
$$\Pi: x \in \mathcal{O} \to \left(\frac{1}{2} |z(\phi) -\Phi(\phi,x)|^2 ,\frac{1}{2} \partial_\phi |z(\phi) -\Phi(\phi,x)|^2\right)$$ is an immersion. Indeed, we notice that
\begin{eqnarray*}  |z(\phi) - \Phi(\phi,x)|^2  &=&  |z_1(\phi) -\Phi_1(\phi,x)|^2 + |z_2(\phi) - \Phi_2(\phi,x)|^2 \\ &=& (\cos \phi - x_1)^2  + (\sin \phi - c - a(\phi) \, (x_2 -c))^2 .\end{eqnarray*}
Therefore,
\begin{eqnarray}\label{E:dz} \frac{1}{2} \partial_\phi |z(\phi) - \Phi(\phi,x)|^2  &=&  (\cos \phi - \Phi_1(x))  (-\sin \phi) \\ & + &\big(\sin \phi -\Phi_2(\phi,x) \big) \big(\cos \phi -a'(\phi) (x_2-c) \big). \nonumber \end{eqnarray}
Hence, 
\begin{eqnarray*} \nabla_x \Big( \frac{1}{2} \partial_\phi |z(\phi) - \Phi(\phi,x)|^2\Big)  &=& \big (\sin \phi, -a(\phi) \cos \phi + a(\phi) a'(\phi) (x_2-c) - a'(\phi) (\sin \phi - \Phi_2(\phi,x) \big).\end{eqnarray*}
On the other hand
\begin{eqnarray*} \nabla_x \Big(\frac{1}{2} |z(\phi) - \Phi(\phi,x)|^2\Big)  &=& -  \big(\cos \phi - \Phi_1(\phi,x), (\sin \phi - \Phi_2(\phi,x)) a(\phi) \big).\end{eqnarray*}
Therefore, 
\begin{eqnarray*} \det J_\Pi(x)  &=& [(\cos \phi -\Phi_1(\phi,x)) \cos \phi + (\sin \phi - \Phi_2(\phi,x)) \sin \phi ] \, a(\phi) \\ &+& [\sin \phi - \Phi_2(\phi,x) - \, a(\phi) (x_2-c)] (\cos \phi -x_1) \,  a'(\phi).\end{eqnarray*}
That is, 
\begin{eqnarray*} \frac{1}{a(\phi)} \det J_\Pi(x)  = \left<z(\phi) - \Phi(\phi,x) , z(\phi) \right> + [\sin \phi -2 \Phi_2(\phi,x)+c] (\cos \phi -x_1) \,  \frac{a'(\phi)}{a(\phi)}.\end{eqnarray*}
Let us show that $J_\Pi(x) \neq 0$ for all $x \in \mathcal{O}$. Indeed, since $\left|\frac{a'(\phi)}{a(\phi)} \right| \leq \frac{1}{(3+|c|) C_\epsilon}$, it suffices to show
$$ |\left< z(\phi) - \Phi(\phi,x), z(\phi) \right> | > \frac{1}{(3+|c|) C_\epsilon}|\sin \phi - 2 \, \Phi_2(\phi,x) +c|  \, |z_1(\phi)- x_1|.$$
This is true because (since $\Phi(\phi,x) \in K$)
$$|\left< z(\phi) - \Phi(\phi,x), z(\phi) \right> | > \frac{1}{C_\epsilon } |z(\phi) - \Phi(\phi,x)| \geq   \frac{1}{C_\epsilon }  |z_1(\phi)- x_1|. $$
and
$|\sin \phi - 2 \, \Phi_2(\phi,x) +c| <3+|c|$. This finishes the first part of the proof.

\medskip

Let us now verify that $\Pi$ is injective. Given $x \neq y \in \mathcal{O}$, we want to show that if $|z(\phi) -\Phi(\phi,x)| = |z(\phi) -\Phi(\phi,y)|$ then
$$\frac{1}{2} \partial_\phi |z(\phi) -\Phi(\phi,x)|^2 \neq \frac{1}{2} \partial_\phi |z(\phi) -\Phi(\phi,y)|^2.$$
Assume by contradiction that the equality holds. We note that (see, e.g., (\ref{E:dz}))
$$\frac{1}{2} \partial_\phi |z(\phi) -\Phi(\phi,x)|^2 = \Big<z(\phi) - \Phi(\phi,x), z(\phi) ^\perp - \big(0,a'(\phi) (x_2-c) \big) \Big>.$$
Therefore,
\begin{eqnarray*}
\left<z(\phi) - \Phi(\phi,x), z(\phi)^\perp - \big(0,a'(\phi) (x_2-c) \big) \right> = \left<z(\phi) -\Phi(\phi,y), z(\phi)^\perp - \big(0,a'(\phi) (y_2-c) \big) \right>.
\end{eqnarray*}
That is,
\begin{eqnarray*}
\left<\Phi(\phi,y) - \Phi(\phi,x), z(\phi)^\perp \right> =a'(\phi) \left[(z_2(\phi)-\Phi_2(\phi,x))(x_2-c) -(z_2(\phi)-\Phi_2(\phi,y)) (y_2-c) \right].
\end{eqnarray*}
We obtain
\begin{eqnarray} \label{E:ineq} \nonumber
\left<\Phi(\phi,y) - \Phi(\phi,x), z(\phi)^\perp \right> &=& a'(\phi)(x_2-y_2) (z_2(\phi)-a(\phi) \, (x_2+y_2 -2c)- c) 
 \\ &=& \frac{a'(\phi)}{a(\phi)}(\Phi_2(\phi,x) -\Phi_2(\phi,y) ) (z_2(\phi)-\Phi_2(\phi,x)-\Phi_2(\phi,y)+c).
\end{eqnarray}
Since $\Phi(\phi,x),\Phi(\phi,y) \in K \subset \mathbb{D}_{1-\eps}$ and $|z(\phi) -\Phi(\phi,x)|= |z(\phi) -\Phi(\phi,y)|$, a simple geometric observation gives
$$|\Phi(\phi,x)-\Phi(\phi,y)| < C_\epsilon \left|\left<\Phi(\phi,x)-\Phi(\phi,y),z(\phi)^\perp \right> \right|\,.$$
In particular, this implies:
$$|\Phi_2(\phi,x)-\Phi_2(\phi,y)| < C_\epsilon \, \left|\left<\Phi(\phi,y)-\Phi(\phi,x), z(\phi)^\perp \right> \right| .$$ Since $\left|\frac{a'(\phi)}{a(\phi)}\right| < \frac{1}{(3+|c|) \, C_\epsilon}$ and $|z_2(\phi)-\Phi_2(\phi,x)-\Phi_2(\phi,y)+c|< 3 +|c|$, the absolute value of the right hand side of (\ref{E:ineq}) is bounded by 
$$ \frac{1}{(3+|c|) C_\epsilon} | \big|\Phi_2(\phi,x) -\Phi_2(\phi,y) \big|  \big|(z_2(\phi)-\Phi_2(\phi,x)-\Phi_2(\phi,y)+c) \big| < \big|\left<\Phi(\phi,y) - \Phi(\phi,x), z(\phi)^\perp \right>\big|.$$
This is a contradiction to (\ref{E:ineq}). We, hence, finish the proof for the second part of the Bolker condition. \end{proof}

We should remark, in passing, that the assumption in Lemma~\ref{L:Cond2} means the object does not change quickly (compared to the rotation of the receiver). However, this restriction is not significant, since for example, if $K \subset \mathbb{D}_{1/2}$ and $c=0$, then the requirement becomes
$$\left|\frac{a'(\phi)}{a(\phi)} \right| \leq \frac{1}{2 \sqrt{3}} \mbox{ for all } \phi \in (\alpha, \beta).$$

\begin{lemma}
Assume one of the following cases:
 \begin{itemize}
 \item[i)] $\beta - \alpha> 2 \pi$ 
 \item[ii)] $[\alpha,\beta] \subset [-\pi/2, 3 \pi/2]$ and
 \begin{equation} \label{E:above} \Phi_2(\phi, x) > \max \{\sin \alpha,\sin \beta\}, \quad \forall x \in K,~\phi \in (\alpha,\beta).\end{equation}
 \end{itemize}
 Then, condition (C) holds for the above vertical stretching model.
\end{lemma}

\begin{proof}
We only prove the theorem under case ii). The other case is easier. Fix $x \in K$, all ellipses passing through $x$ are of the form
$$\mathcal{E}_\phi= \{y: |z(\phi) -\Phi(\phi,x)| = |z(\phi) -\Phi(\phi,y)|\}, \quad \phi \in (\alpha,\beta).$$
The outward unit normal vector of $\mathcal{E}_\phi$ at $x$ has the same direction as
\begin{eqnarray*} \nu_\phi &=& \big(z_1(\phi)-\Phi_1(\phi, x), a(\phi) \, (z_2(\phi)-\Phi_2(\phi, x)) \big).\end{eqnarray*}
It now suffices to show that the conic set
$$V= \{ t \nu_\phi: t \geq 0, \phi \in (\alpha, \beta)\}$$ contains the (closed) upper half plane. Indeed, we observe that $V$ is the upper solid angle whose boundary is the union of two rays:
$$\{t(\cos \alpha - x_1, a(\alpha) (\sin \alpha -x_2) ): t \geq 0\}$$
and 
$$\{t(\cos \beta - x_1, a(\beta) (\sin \beta -x_2)): t \geq 0\}.$$
Since, see (\ref{E:above}), $$\Phi_2(\phi,x)>\max\{\sin(\alpha), \sin(\beta)\}, \quad \mbox{ for all } \phi \in (\alpha, \beta),$$ the above two rays are in the lower half plane. We hence conclude that $V$ contains the upper half plane. This finishes our proof. \end{proof}

\section{Uniqueness results} \label{S:unique}
In this section, we prove two uniqueness results for the continuous model. The first one is the uniqueness of Problem (P) when the deformation family $\Phi$ is known. This result holds for any general family of deformation satisfying conditions (A.1-3), (B), and (C). The second result is the injectivity of the linearized problem of finding $\Phi$ given the function $f$. This result is only proved for the vertical stretching model introduced in Section~\ref{Ss:stretching}. Our approach is based on the theory of analytic wave front set (see, e.g., \cite{Ho1}) and, interestingly, convex analysis. 

Our use of the analytic wave front set is inspired by several works in integral geometry, such as \cite{boman1987support,quinto1994radon,frigyik2008x,HomanZhou}. Let us now briefly recall some basic facts of that theory. The following definition of analytic wave front set follows from \cite[Definition~6.1]{sjostrand1985singularites} (which is in spirit of Bros-Iagolnitzer \cite{bros1975support}).
Let us denote by $\mathbb{T}^*\mathbb{R}^n \setminus 0$ the cotangent bundle of $\mathbb{R}^n$ excluding the zero section. For simplicity, it can be identify with $\mathbb{R}^n \times (\mathbb{R}^n \setminus 0)$. Let $(x_0,\xi_0) \in \mathbb{T}^* \mathbb{R}^n \setminus 0$ and $\psi(x,y,\xi)$ be an analytic function defined in a neighborhood $U$ of $(x_0,x_0,\xi_0) \in \mathbb{C}^{3n}$ such that
\begin{itemize}
\item[1)] For all $(x,x,\xi) \in U$, (i.e., $x=y$), we have
$\psi(x,x,\xi) = 0$ and $\partial_x \psi(x,x,\xi)= \xi$.
\item[2)] There exists $c>0$ such that for all $(x,y,\xi) \in U$, we have
$$\Im \, \psi(x,y,\xi) \geq c \, |x-y|^2.$$
Here, $\Im \, \psi$ denotes the imaginary part of $\psi$.
\end{itemize}

Let $a(x,y,\xi,\lambda)$ be an elliptic classical analytic symbol defined on $U$. We say that a distribution $u$ in $\mathbb{R}^n$ is analytic microlocally near $(x_0,\xi_0)$ if there exists a cut-off function $\chi \in C^\infty(\mathbb{R}^n)$ with $\chi(x_0)=1$ such that
$$\int e^{i \,\lambda \, \psi(x,y,\xi)} \, a(x,y,\xi,\lambda) \chi(y) \, \overline u(y) \, dy = \mathcal{O}(e^{-\lambda/C}).$$
for some $C>0$, uniformly in a conic neighborhood of  $(x_0,\xi_0)$.  The analytic wavefront set is the closed conic set $\wf_A(u) \subset \mathbb{T}^* \mathbb{R}^n$ which is the complement of the set of covectors near which $u$ is microlocally analytic.

\medskip

Let us state the first result of this section.
\begin{theorem}
	\label{thm:1}
Assume that the family of deformations $\Phi(\phi,.)$ is known and satisfies conditions (A.1-3), (B), and (C). Then, $f$ is uniquely determined from the data $\mathcal{R}(f_\phi)=\mathcal{R}_\Phi(f)$. 
\end{theorem}

\begin{proof} Let us recall that
$$\mathcal{R}_\Phi(f)(\phi,r)= \int_{\mathcal{E}(\phi,r)} w(\phi,x) \, f(x)  \,d \sigma(x),$$ where $$w(\phi,x)=\frac{J_x[\Phi(\phi,x)]  \, |z(\phi) -\Phi(\phi,x)|}{|\left< z(\phi) - \Phi(\phi,x), \nabla_x \Phi(\phi,x) \right> |}>0$$ is an analytic function. Then, $\mathcal{R}_\Phi$ is a Radon transform satisfying the Bolker condition, which has been investigated in \cite{quinto1994radon,HomanZhou}. Assume that $\mathcal{R}_\Phi(f)(\phi,r)=0$ for all $(\phi,r) \in (\alpha,\beta) \times \mathbb{R}_+$. Then, (see, \cite[Proposition~3.1]{quinto1994radon} or \cite[Proposition 1]{HomanZhou}), $$(x,\xi) \not \in \wf_A(f) \mbox{ for all } (x,\xi) \mbox{ being conormal to a curve } \mathcal{E}(\phi,r).$$ 
From condition (C), we obtain $(x,\xi) \not \in \wf_A(f)$ for all $(x,\xi) \in \mathbb{T}^*\mathcal{O}$. That is, $f$ is an analytic function (see, e.g., \cite[Theorem~8.4.5]{Ho1}). Since $f$ is compactly supported, it has to be the zero function. 
\end{proof}

\medskip

We now study the linearization of Problem (P) with respect to the motion. We recall from (\ref{E:data}) that \begin{eqnarray*}  \mathcal{R}_\Phi(f)(\phi,r) = \mathcal{R}(f_\phi)(\phi,r)  &=& \int_{\mathbb{R}} \int_{\mathbb{R}^2} e^{i \lambda (|z(\phi) - \Phi(\phi,x)|-r )} J_x[\Phi(\phi,x)] \, f(x) \, dx \, d \lambda .
\end{eqnarray*}

Let us linearize the operator $\mathcal{R}_\Phi$ with respect to $\Phi$ at the background $\Phi^*$ along the direction $\bd=\bd(\phi,x)$ as follows,
\begin{eqnarray*} D_f(\bd)(\phi,r) &=& \frac{d}{d\epsilon} \int\limits_{\mathbb{R}}\int\limits_{\mathbb{R}^2} e^{i(|z(\phi) -(\Phi^*(\phi,x)+\epsilon \delta(\phi,x))| -r ) \lambda} J_x (\Phi^*(\phi,x) + \epsilon  \bd(\phi,x)) \, f(x) \, dx \, d\lambda \Big|_{\epsilon=0}.
\end{eqnarray*}
That is,
\begin{eqnarray}\label{E:Lin} D_f(\bd)(\phi,r)  &=& \int\limits_{\mathbb{R}} \int\limits_{\mathbb{R}^2} e^{i(|z(\phi) -\Phi^*(\phi,x) | - r) \lambda} \, p(z,x,\lambda) \, f(x) \, dx \, d\lambda,
\end{eqnarray}
where
$$p(z,x, \lambda) =  i \lambda \frac{\left<z(\phi) -\Phi^*(\phi,x),  \bd(\phi,x) \right>}{|z(\phi) -\Phi^*(\phi,x)|} \, J_x[\Phi^*(\phi,x)]+ \frac{d}{d \epsilon}J_x (\Phi^*(\phi,x) + \epsilon \bd(\phi,x)) \big|_{\epsilon=0}.$$

Let us consider the vertical stretching model introduced in Section~\ref{Ss:stretching}:
$$\Phi(\phi,x) = (x_1, c+ a(\phi) (x_2-c) ).$$ Then, we can assume
$$\Phi^*(\phi,x) = (x_1, c+ a^*(\phi) (x_2-c) ),$$
and 
$$\bd (\phi,x) = (0, d(\phi) \, (x_2-c)),$$ where $d(\phi) =a(\phi) - a^*(\phi)$.
Let us denote $$\nu_2(\phi,x) =  \frac{z_2(\phi)-\Phi^*_2(\phi,x)}{|z(\phi) -\Phi^*(\phi,x)|}.$$
Then, direct calculations show
$$p(z,x, \lambda) =  \big [ i \lambda \nu_2(\phi,x) \,a^*(\phi) (x_2-c) + 1\big] d(\phi).$$
From equation (\ref{E:Lin}), we obtain
\begin{eqnarray*} D_f (\bd)(\phi,r)  &=& d(\phi) \, \mathcal{T}(f)(\phi,r),
\end{eqnarray*}
where
\begin{eqnarray*}\label{E:zero} \mathcal{T}(f)(\phi,r) = \int\limits_{\mathbb{R}} \int\limits_{\mathbb{R}^2} e^{i(|z(\phi) -\Phi^*(\phi,x)|-r) \lambda} \, \big[ i \lambda \nu_2(\phi,x)  \,a^*(\phi) \,  (x_2-c) +1 \big] \,  f(x) \, dx \, d\lambda.
\end{eqnarray*}
Let us observe that $\mathcal{T}$ is a geometric integral transform on the family of curves
$$\mathcal{E}_*(\phi,r) = \{x \in \mathbb{R}^2: |z(\phi) - \Phi^*(\phi,x) |=r \}.$$
Unlike $\mathcal{R}_\Phi(f) $, $\mathcal{T}$ not only involves the function $f$ but also its first derivative (since the amplitude function is of order $1$). Moreover, it also involves vanishing weight. 

To analyze $D_f$, we will also assume that $\Phi^*$ satisfies conditions (A.1-3), (B), and (C). We also assume that $a(\phi)$ is analytic. We will need the following result
\begin{lemma} \label{L:ellip}  Let $(x_0,z_0,r_0)$ be such that $$(x_{0,2}-c)  \, \nu_2(\phi_0,x_0) \neq 0$$ and $\mT(f)(\phi,r) =0$ in a neighborhood of $(z_0,r_0)$. Then,
 $$(x_0, \xi_0) \not \in \wf_A(f),$$ if $(x_0,\xi_0)$ is conormal to $\mathcal{E}_*(z_0,r_0).$
\end{lemma}

The proof of Lemma~\ref{L:ellip} follows closely the arguments in \cite{quinto1994radon,frigyik2008x,HomanZhou} and is omitted for brevity. Let us now state the second main result of this section.


\begin{theorem} \label{T:Linearized}
Consider the vertical stretching model. Assume $\supp(f)$ lies in the open upper half plane of the line $x=c$ and its boundary does not contain any vertical line segments. Assume further that $\{\Psi^*(\phi,z(\phi)): \phi \in (\alpha,\beta)\}$ does not contain any horizontal line segments. Then, $D_f$ is injective.
\end{theorem}

In order to prove this theorem, we will need some basic knowledge of convex analysis. Let us recall that a function $h$ defined in an interval $\mathbb{I}$ is convex if 
$$h(\alpha \tau + (1-\tau)\tau') \leq  \alpha \, h(\tau) + (1-\tau) \, h(\tau'), \mbox{ for all } 0 \leq \alpha \leq 1, ~\tau, \tau' \in \mathbb{I}.$$
For each $\tau_0$ in the interior of $\mathbb{I}$, a convex function $h$ admits a finite left-derivative and a finite right-derivative (see, e.g., \cite[Theorem~4.1.1]{hiriart1993convex}):
\begin{eqnarray*} 
D_{-}h(\tau) = \lim_{\tau \to \tau_0^-} \frac{h(\tau) - h(\tau_0)}{\tau -\tau_0}, \\
D_{+}h(\tau) = \lim_{\tau \to \tau_0^+} \frac{h(\tau) - h(\tau_0)}{\tau -\tau_0},
\end{eqnarray*}
respectively that satisfy
$$D_{-}h(\tau_0) \leq D_{+} h(\tau_0). $$
The set $$\partial h(\tau_0) = [D_{-}h(\tau_0) , D_{+} h(\tau_0)]$$ is called the subdifferential of $h$ at $\tau_0$. In particular, if $\partial h(\tau_0)$ is a singleton then $h$ is differentiable at $\tau_0$. 

Furthermore, we will need some results from microlocal analysis of analytic wave front set. Let $F$ be a closed subset of $\mathbb{R}^n$. The exterior conormal $\mathbb{N}_e(F)$ is defined to be the set of all $(x_0,\xi_0)$ such that $x_0 \in F$ and there exists a function $f \in C^2(\mathbb{R}^n)$ with $\nabla f(x_0)=\xi_0 \neq 0$ and 
$$f(x) \leq f(x_0), \mbox{ when } x \in F \cap U,$$ where $U$ is some open neighborhood of $x_0$. We also denote 
$$\mathbb{N}(F) = \mathbb{N}_e(F) \cup ( - \mathbb{N}_e(F)). $$
We will make use of the following result (see, e.g., \cite[Theorem~8.5.6']{Ho1})
\begin{equation} \label{E:Hormander} \mathbb{N}( \supp\,  u) \subset \wf_A(u)\,,\end{equation}
and the Kashiwara's Watermelon theorem (see, e.g., \cite[Theorem~2.1]{hormander1993remarks}), stated as follows.

\begin{theorem}
Let $u$ be a distribution in $\mathbb{R}^n$ and $(x_0,\xi_0) \in \mathbb{N}(\supp\, u)$, then
\begin{equation} (x_0,\xi) \in \wf_A(u)  \Rightarrow (x_0, \xi + t \xi_0) \in \wf_A(u), \mbox{ for all } t \in \mathbb{R} \mbox{ with }  \xi + t \xi_0 \neq 0.
\end{equation}
\end{theorem}

%

We are now ready to prove Theorem~\ref{T:Linearized}.

\begin{proof} Let us assume, by contradiction, that there exists  $d \not \equiv 0$ such that $D_f(0,d)(\phi,r) \equiv 0$.  Then, there exists an open interval $\mathcal{A} \subset (\alpha, \beta)$ such that  $d(\phi) \not =0$ for all $\phi \in \mathcal{A}$. We obtain
$$\mathcal{T}(f)(\phi,r) = 0, \mbox{ for all } \phi \in \mathcal{A} \mbox{ and } r >0. $$
Let us recall 
$$\mathcal{E}_*(\phi,r) = \{x: |z(\phi)  -\Phi(\phi,x)|=r\}.$$ 
For each $\phi \in \mathcal{A}$, let $r_\phi>0$ be the smallest value $r$ such that $\mathcal{E}_*(\phi,r)$ intersects $\supp(f)$. Let $x_\phi$ be such an intersection point and $\xi_\phi$ be orthogonal to $\mathcal{E}_*(\phi,r_\phi)$ at $x_\phi$. Then, $(x_\phi,\xi_\phi) \in \mathbb{N}(\supp (f))$. 
Using (\ref{E:Hormander}), we obtain $$(x_\phi,\xi_\phi) \in \wf_A(f).$$
Due to Lemma~\ref{L:ellip} and $x_{\phi,2} > c$, $$\nu_2(\phi,x_\phi) = 0. $$
From the definition of $\nu(\phi,x_\phi)$, we get $$x_{\phi,2} =\Psi^*_2(\phi,z(\phi)) .$$
On the other hand, $\mathcal{E}_*(\phi,r_\phi) $ is an ellipse centered at $\Psi^*(\phi,z(\phi))$. Therefore, the normal vector of $\mathcal{E}_*(\phi,r_\phi)$ at $x_\phi$ is horizontal.  That is, $\xi_\phi$ is a horizontal vector and $x_\phi$ is unique (since it belongs to the same horizontal line as $\Psi^*(\phi,z(\phi))$).

Due to the assumption that the curve $\{\Psi^*(\phi,z(\phi)): \phi \in (\alpha,\beta)\}$ does not contain any piece of horizontal straight line, by shrinking $\mathcal{A}$ if necessary, we can assume 
$$\Psi^*_2(\phi,z(\phi)) \neq \Psi^*_2(\phi',z(\phi'))$$ for all $\phi \neq \phi'$ belonging to $\mathcal{A}$. 
Let us denote $$\mathbb{I}=\{x_{\phi,2}=  \Psi^*_2(\phi,z): \phi \in \mathcal{A}\}.$$ 
Then $\mathbb{I}$ is a nonempty open interval and the following function is well-defined
\begin{eqnarray*} && h: \mathbb{I} \longrightarrow \mathbb{R},\\[6 pt] && x_{\phi,2} \longmapsto  x_{\phi,1}.\end{eqnarray*} 
Let $\mS$ be the swapping operator $\mS(x_1,x_2) = (x_2,x_1)$. Then, $$(\tau,h(\tau)) \in \mS[\partial \supp(f)], \mbox{ for all } \tau \in \mathbb{I}.$$ 
That is, $\graph(h) \subset \mS[\partial \supp(f)]$. We now prove that $h$ is a constant function. This implies $\partial \supp(f)$ contains a vertical straight line segment, which is a contradiction to our assumption and the theorem is proved. To proceed with the proof, let us notice that for each $\phi \in \mathcal{A}$, $(x_\phi,\xi_\phi) \in \mathbb{N}(\supp(f))$. If $h$ is differentiable then $\mS(\xi_\phi)$ is orthogonal to the tangent line of the graph of $h$. Since $\mS(\xi_\phi)$ is vertical for all $\phi \in \mathcal{A}$, we conclude that $h$ is constant on $\mI$. It, therefore, now suffices to prove that $h$ is differentiable. Our approach is to ``convexify" the function $h$ in order to make use of the theory of convex analysis.  
\medskip

Let $\tau_0 \in \mI$, we now prove the following claim: there exist $\epsilon>0$ such that $$h^*(\tau) := h(\tau) + \frac{1}{\epsilon} \, \tau^2$$ is convex on $\mI_{\tau_0,\epsilon} = (\tau_0 -\epsilon, \tau_0+\epsilon)$. Indeed, we recall that for each $\tau \in \mI$, the ellipse $\mS[\mE_*(\phi,r_\phi)]$ does not intersect the interior of the epigraph of $h$. For $\epsilon$ small enough and $\tau \in \mI_{\tau_0,\epsilon}$, there is a function $h_\tau: \mI_{\tau_0,\epsilon} \to \mathbb{R}$ such that: 
\begin{itemize} 
\item[i)] $\graph(h_\tau) \subset \mS[\mE_*(\phi,r_\phi)]$, and 
\item[ii)]  $h_\tau''(s) \geq -\frac{1}{\epsilon}$ on $\mI_{\tau_0,\epsilon}$. 
\end{itemize}
We define $h_\tau^*(s) = h_\tau(s) + \frac{1}{\epsilon} s^2$, then $h^*_\tau$ is convex on $\mI_{\tau_0,\epsilon}$. It is easy to see that
$$h^*=\sup_{\tau \in \mI_{\tau_0,\epsilon}} h_\tau^*.$$ Therefore, $h^*$ is a convex function on $\mI_{\tau_0,\epsilon}$. 

\medskip

Let us now prove that $h^*$ is differentiable at $\tau_0$, or equivalently, $\partial h^*(\tau_0)$ is a singleton. To this end, let us define  
\begin{equation*}
	F(x) =\left(x_1,x_2- \frac{1}{\epsilon} x_1^2\right),\quad \mbox{ and } f^*(x) = f(F(x)).
\end{equation*}
Then, $\graph(h^*) \subset \mS[\partial \supp(f^*)]$. We notice that for each $s \in \partial h^*(\tau_0)$, $$(-1,s) \in \mathbb{N}(\supp(f^*)).$$ Let assume $\partial h^*(\tau_0)$ is not a singleton. Then $\supp(f^*)$ has at least two linearly independent exterior normal vectors $\theta_1,\theta_2$ at $y=(h^*(\tau_0),\tau_0)$. Applying the Kashiwara's Watermelon theorem, we obtain $$(y,\eta) \in \wf_A(f^*) \mbox{ for all } \eta \in \mathbb{R}^2 \setminus 0.$$ Let $x= F(y) \in \supp(f)$. Since $F$ is an analytic bijection from $\mathbb{R}^2$ to $\mathbb{R}^2$, \cite[Theorem~8.5.1]{Ho1} gives \begin{equation} \label{E:wfA} (x,\xi) \in \wf_A(f) \mbox{ for all } \xi \neq 0.\end{equation} We notice that $x=(h(\tau_0),\tau_0)$. Let $\mathbb{I} \ni \tau \neq \tau_0$ and $\phi \in \mathcal{A}$ such that $\pi_2( \Psi(\phi,z(\phi)) )= \tau$. There is an $r>0$ such that $\mathcal{E}(\phi,r)$ passes through $x$. Let $\xi = \nabla_x |z(\phi) - \Phi^*(\phi,x)|$, which is a normal vector of $\mathcal{E}(\phi,r)$ at $x$. Since $\Psi^*_2 (\phi,z(\phi)) \neq \Psi^*_2(\phi_0,z(\phi_0)) = x_2$, we have $$z_2(\phi) - \Phi^*_2(\phi,x) = a^*(\phi) (\Psi_2^*(\phi,z(\phi)) - x_2) \neq 0.$$ Therefore, $\nu_2(\phi,x) \neq 0$. From Lemma~\ref{L:ellip}, we obtain (noticing that $x_2>c$ since $x \in \supp(f)$)
$$(x,\xi) \not \in \wf_A(f).$$
This is a contradiction to (\ref{E:wfA}).  Thus, the theorem is proved.
\end{proof}

We remark here some geometric restrictions in Theorem~\ref{T:Linearized} may be relaxed (or even removed). That can be done by obtaining a better version of Lemma~\ref{L:ellip}, which is a topic of upcoming research. 


In summary, we have formulated an approximate continuous model for the problem of PAT reconstruction with motion and proved two uniqueness results.  In particular, we showed that when the family of deformations is known, the desired function $f$ is unique (Theorem~\ref{thm:1}), and we showed a one-to-one correspondence of the linearized problem for obtaining $\Phi$ given $f$, under mild assumptions (Theorem~\ref{T:Linearized}). These two results indicate that the reconstruction for PAT with motion may be a stable problem, at least for the vertical stretching model presented in Section~\ref{Ss:stretching}. In the next section, we will demonstrate this numerically. Namely, we consider the discretized problem and describe numerical methods for simultaneous motion estimation and PAT image reconstruction.
%
%

\section{Computational approach for simultaneous motion estimation and image reconstruction} 
\label{sec:methods}
In this section, we describe a computational approach to simultaneously estimate motion parameters and obtain a reconstruction for the discrete PAT problem.  We first recast the problem as a \emph{separable} nonlinear least squares problem.  Then we use variable projection methods  \cite{GoPe73,GoPe03,Kau75,RuWe80,o2013variable} to efficiently solve the problem by exploiting structure in the variables, i.e., separating the linear and nonlinear variables.  The work presented here builds on work presented in \cite{ChHaNa06,Chung2010b}, but a significant contribution is our application to PAT reconstruction problems under motion, where the particular forms of the forward model and the deformation model have not been considered in this framework before.  Previous work on using variable projection to improve PAT reconstruction for transducers with inaccurate electric impulse responses can be found in \cite{sheng2015constrained}. For computational efficiency, we take advantage of matrix sparsity and object-oriented programming for implementing projection operations, and we use a hybrid LSQR method for automatic regularization parameter selection \cite{ChNaOLe08}.


First, we formulate the discrete PAT reconstruction problem with motion.  Assume $c(x)$ is constant, then following the notation in Section~\ref{sec:continuous}, let $\bff\in \bbR^N$ be the discretized desired solution\footnote{Here we assume the 2D image is vectorized column-wise.} and let $\left\{z_i\right\}, i = 1, ...n$ denote the locations of the transducers.   Let us consider the discrete version of the deformation model in Section~\ref{Ss:stretching}. That is, we assume vertical stretching, which is a relatively realistic model for motion resulting from breathing. Let $a(z_i) = 1 + \gamma_i$ be the stretch factor corresponding to the transducer at $z_i$, so the $i$th transformed image is given by $\bff_i = \bfK(\gamma_i) \bff,$ where $\bfK(\gamma_i) \in \bbR^{N \times N}$ performs one dimensional stretching/expansion as described in Section~\ref{Ss:stretching}.  We use bilinear interpolation to determine the weights in sparse matrix $\bfK$.

At each transducer location $z_i$, assume there are $m$ radii and let $\bfA_i \in \bbR^{m \times N}$ be the corresponding projection matrix for that location (i.e., $\bfA_i \bff$ is the discrete spherical Radon transform of $\bff$ on circles centered at $z_i$).
Thus the observed spherical projection measurements for all $m$ radii are contained in vector
$$\bfg_i = \bfA_i \, \bfK(\gamma_i) \, \bff + \bfe_i \,\, \in \bbR^{m}$$
where $\bfe_i \in \bbR^{m}$ is additive Gaussian noise that is independent and identically distributed.
Let $\bfgamma =\begin{bmatrix}\gamma_1 & \cdots &\gamma_n \end{bmatrix}\t \in \bbR^n$ and define
$$\bfg =\begin{bmatrix} \bfg_1 \\ \vdots \\ \bfg_n
\end{bmatrix}\in \bbR^{nm}, \quad \bfA(\bfgamma) = \begin{bmatrix}
	 \bfA_1 \bfK(\gamma_1)\\ \vdots \\ \bfA_n \bfK(\gamma_n)
\end{bmatrix}\in \bbR^{mn \times N}, \quad \mbox{and} \quad \bfe = \begin{bmatrix}\bfe_1 \\ \vdots \\ \bfe_n
\end{bmatrix}\in\bbR^{mn}$$
then, the discrete mathematical model for PAT reconstruction with motion is given by
\begin{equation}
	\label{eqn:model}
	\bfg = \bfA(\bfgamma)\bff + \bfe\,,
\end{equation}
where the goal is to estimate the desired image $\bff$ as well as the motion parameters $\bfgamma$, given the observations $\bfg.$
In Figure~\ref{fig:PATproblem} we provide a sample desired image $\bff$ which has $256 \times 256$ pixels and corresponding PAT observations $\bfg$ where transducers are located at $120$ equidistant angles between $0$ and $357$ at $3$ degree intervals and each projection corresponds to $m=363$ radii.  The test image used here was obtained from the Auckland MRI Research Group website \cite{AukMRI}. It is supported inside the square $[-\frac{1}{2}, \frac{1}{2}] \times [-\frac{1}{2}, \frac{1}{2}]$. In our numerical experiment, the base line of the vertical stretch is $x_2 = -\frac{1}{2}$ and is assumed known. An illustration of vertical stretching is provided in Figure~\ref{fig:motionillustrate}.
\begin{figure}[bthp]\footnotesize
\begin{center}
	\begin{tabular}{cc}
		\includegraphics[width=1.5in]{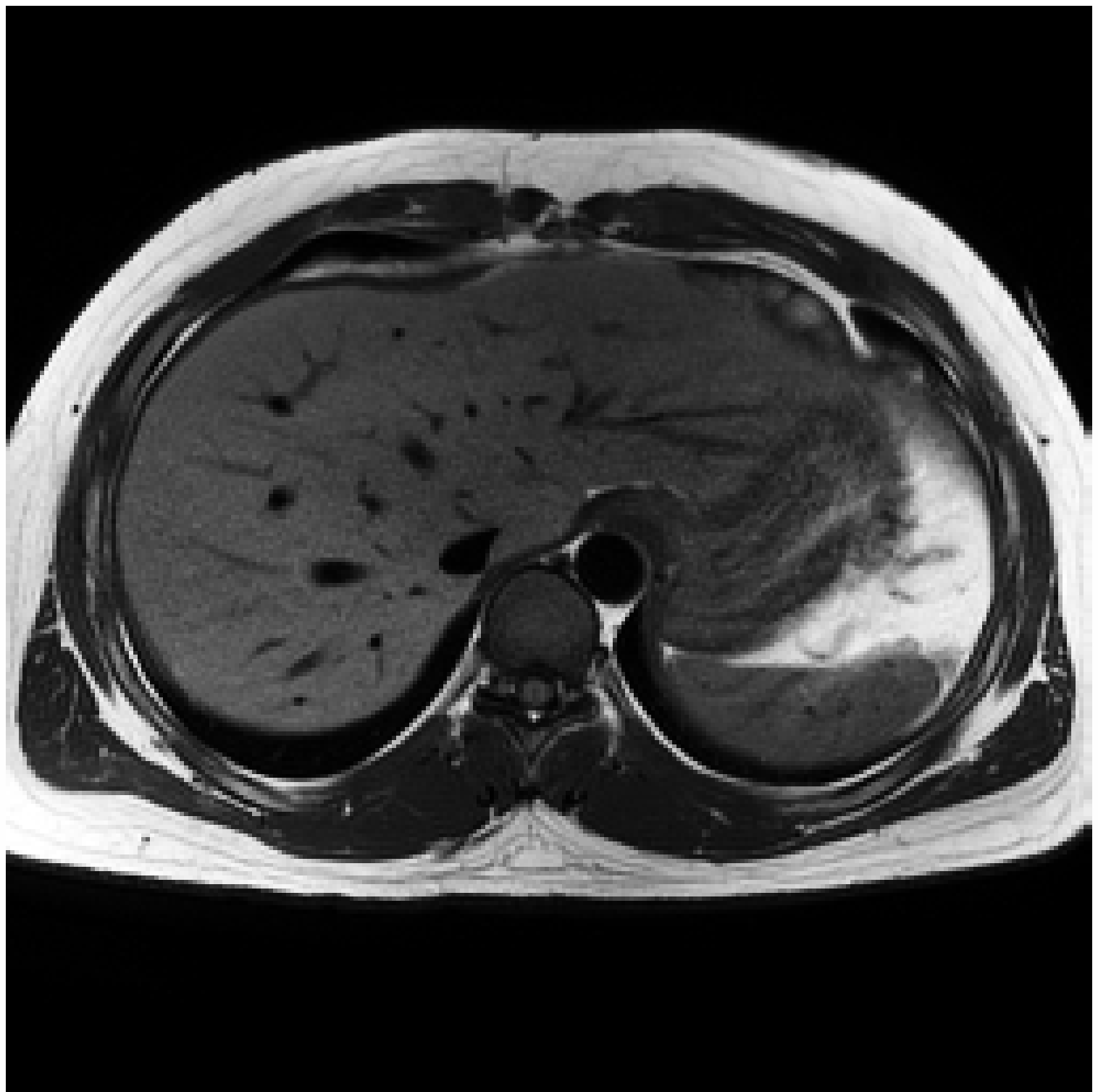} &
						\includegraphics[width=1in, angle =90]{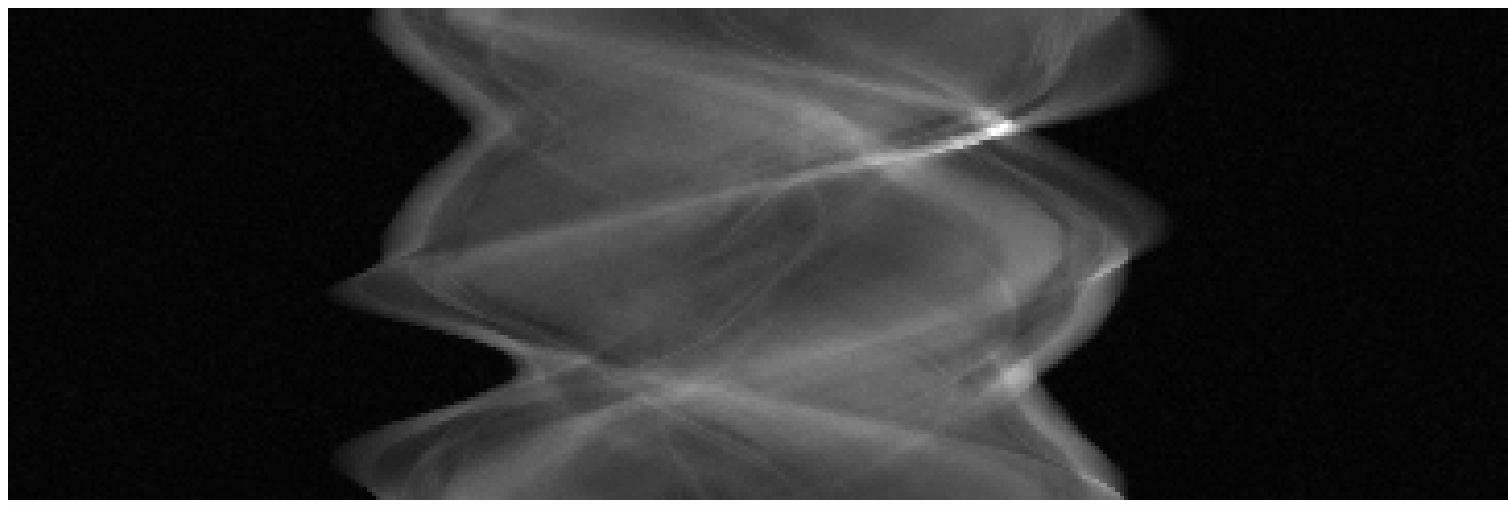}  \\
	(a) Desired image & (b) Observations in sinogram
	\end{tabular}
\end{center}
\caption{In 2D PAT reconstruction, the goal is to reconstruct image in (a) from observed measurements shown in (b). Each row of the sinogram image corresponds to a transducer location, where the columns correspond to different radii.  Here, the sinogram image is $120 \times 363$.}
	\label{fig:PATproblem}
\end{figure}

\begin{figure}
	\includegraphics[width = 6in]{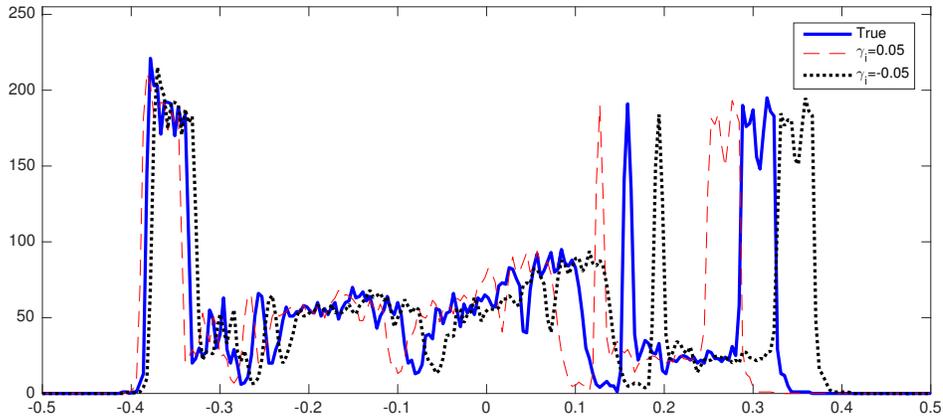}\\
\label{fig:motionillustrate}
\caption{Illustration of vertical stretching.  The solid line corresponds to pixel values contained in column $100$ of the true image shown in Figure~\ref{fig:PATproblem}(a). The dashed and dotted lines correspond to pixel values in column $100$ of deformed images $\bfK(\gamma_i) \bff$ (after reshape) for $\gamma_i = 0.05$ and $\gamma_i = -0.05$ respectively.  Notice that features that are farther from the base line are stretched more. }
\end{figure}

If the motion parameters $\bfgamma$ are fixed, then estimating the image $\bff$ is a linear inverse problem, and various algebraic methods can be used for reconstruction.  In particular, the least squares (LS) problem,
\begin{equation}
	\label{eqn:lin}
	\min_\bff \norm[2]{\bfA(\bfgamma)\bff - \bfg }^2
\end{equation}
gives one solution, but as we will see in Section~\ref{sub:solving_the_linear_reconstruction_problem}, the LS solution is poor, and regularization is required to stabilize the solution.  Regardless, if the motion parameters are not known precisely or only estimated from the data, then the problem becomes simultaneous estimation of the motion parameters and the desired image.  In this paper, we consider variable projection methods to solve the separable nonlinear least squares problem,
\begin{equation}\label{eq:sepLS}
	\min_{\bff, \bfgamma} \norm[2]{\bfA(\bfgamma) \bff -\bfg}^2\,.
\end{equation}
There are a few important features that make this problem ideal for variable projection methods.  First, for a parameterized motion model (e.g., vertical stretching) that is a good approximation to the actual motion, the number of parameters in $\gamma$ is typically much smaller than the number of unknowns in $\bff$ (i.e., $n \ll N$).  Furthermore the residual function in~\eqref{eq:sepLS} is linear in $\bff$ and nonlinear in $\bfgamma$, so we can exploit separability.  In many applications, such features are exploited by decoupling the problems. That is, researchers first seek a good set of parameters $\bfgamma$ by using sophisticated image registration or calibration techniques, and then solve a linear inverse problem for $\bff$.  Although this approach (or an approach that alternates between the two steps) may work in some scenarios, convergence to a solution can be slow \cite{ChHaNa06}.  On the other hand, a fully coupled approach that simultaneously optimizes over $\bff$ and $\bfgamma$ could be used, but being able to obtain a good estimate of $\bfgamma$ may be prohibitively expensive, especially if the variables are tightly coupled \cite{ChHaNa06}. Thus, in this paper, we consider a partially coupled framework that can take algorithmic advantage of the structure of the problem and handle nonlinearity directly.

To begin, we describe some of the challenges and approaches to solve the linear reconstruction problem~\eqref{eqn:lin} (when $\bfgamma$ is fixed) in Section~\ref{sub:solving_the_linear_reconstruction_problem}.  Then, in Section~\ref{sub:variable_projection_for_simultaneous_estimation}, we describe a variable projection approach to solve~\eqref{eq:sepLS} for simultaneous PAT motion estimation and reconstruction.

\subsection{Regularization for the linear problem} 
\label{sub:solving_the_linear_reconstruction_problem}
Assume motion parameters $\bfgamma$ are fixed. Even if accurate estimates of parameters $\bfgamma$ are available, the linear inverse problem~\eqref{eqn:model} where $\bfA = \bfA(\bfgamma)$ can be very difficult to solve.  This is because the underlying problem is \emph{ill-posed} \cite{Hadamard1923}, where a main challenge is that small errors in the data may result in large errors in the reconstruction.  

In order to obtain a meaningful reconstruction, regularization is needed to stabilize the inversion process.  The basic idea of regularization is to impose prior knowledge about the noise in the data and about the unknown solution \cite{Kaipio2004,Calvetti2011,Mueller2012,Hansen2010,Bertero1998,Engl2000,Vogel2002,HaHa93}.
Here we consider Tikhonov regularization \cite{phillips1962technique,tikhonov1977solution}, where the goal is to solve regularized problem,
\begin{equation}
	\min_{\bff} \frac{1}{2}\| \bfA \bff - \bfg \|_2^2 + \lambda^2 \norm[2]{\bff}^2
\end{equation}
where $\lambda$ is a regularization parameter that balances the tradeoff between the data fit term and the regularization term.  Selecting a good regularization parameter can be a challenging and delicate task, especially for large-scale problems.  Various techniques have been studied in the literature \cite{Hansen2010}, and it is beyond the scope of this paper to provide a complete review of such methods.  In particular, we use hybrid LSQR methods that can solve large linear least squares problems and automatically estimate regularization parameters for Tikhonov.  Below we provide an illustration, for readers not familiar with hybrid iterative methods.

\paragraph{Illustration of hybrid methods}
Hybrid iterative methods have been proposed as a means to compute regularized solutions to large-scale inverse problems.  In this illustration, we demonstrate the use of hybrid methods for the linear PAT reconstruction problem.  Let $\bfA = \bfA(\bfgamma^\true)$ and consider using LSQR on the unregularized problem.  With no additional regularization (other than early termination of the iterative method), it is known that semiconvergence will occur, whereby early reconstructions tend to provide improved solution reconstructions, but noise dominates later reconstructions.  This phenomenon is evident in the relative error plot (see Figure \ref{fig:linear} for relative reconstruction errors, computed as $\norm[2]{\bff^\true - \bff^{(k)}}/\norm[2]{\bff^\true}$ where $\bff^{(k)}$ is the reconstruction at the $k$th iteration).  Hybrid methods combine an iterative scheme such as the Golub-Kahan bidiagonalization with a direct regularization method such as Tikhonov regularization.  Since regularization is performed on the projected problem, hybrid methods can overcome semiconvergence.  Furthermore, sophisticated regularization parameter selection methods can be used for the projected problem.  In this paper, we use the hybrid LSQR implementation, HyBR, described in \cite{ChNaOLe08}, which provides a Tikhonov-regularized solution where the regularization parameter is automatically selected using a weighted-GCV method. In Figure~\ref{fig:linear}, we provide relative reconstruction errors for HyBR, as well as for HyBR where the optimal regularization parameter is used. This parameter corresponds to minimal reconstruction error and is not obtainable in practice.  For this problem, HyBR performs well, where the selected regularization parameter is $\lambda_{\rm WGCV} = .0389$ at iteration $100$, and the corresponding reconstruction is provided in Figure~\ref{fig:linearimages}(a).  In summary, for large problems such as PAT where it may be difficult to obtain a good regularization parameter a priori, hybrid methods can be good for obtaining regularized solutions for linear problems.

\begin{figure}[bthp]
\begin{center}
		\includegraphics[width=5.5in]{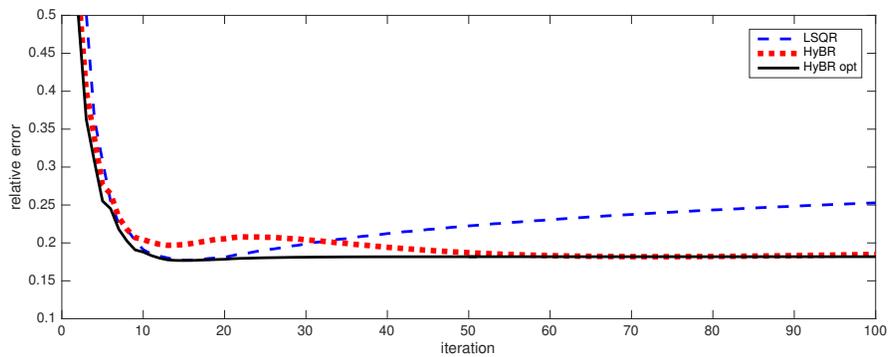} 
\end{center}
\caption{Relative reconstruction errors for the linear PAT reconstruction problem corresponding to LSQR (no regularization) and HyBR with automatic weighted-GCV-selected and the optimal regularization parameter.}
	\label{fig:linear}
\end{figure}

\begin{figure}[bthp]
\begin{center}
	\begin{tabular}{cc}
		\includegraphics[width=1.6in]{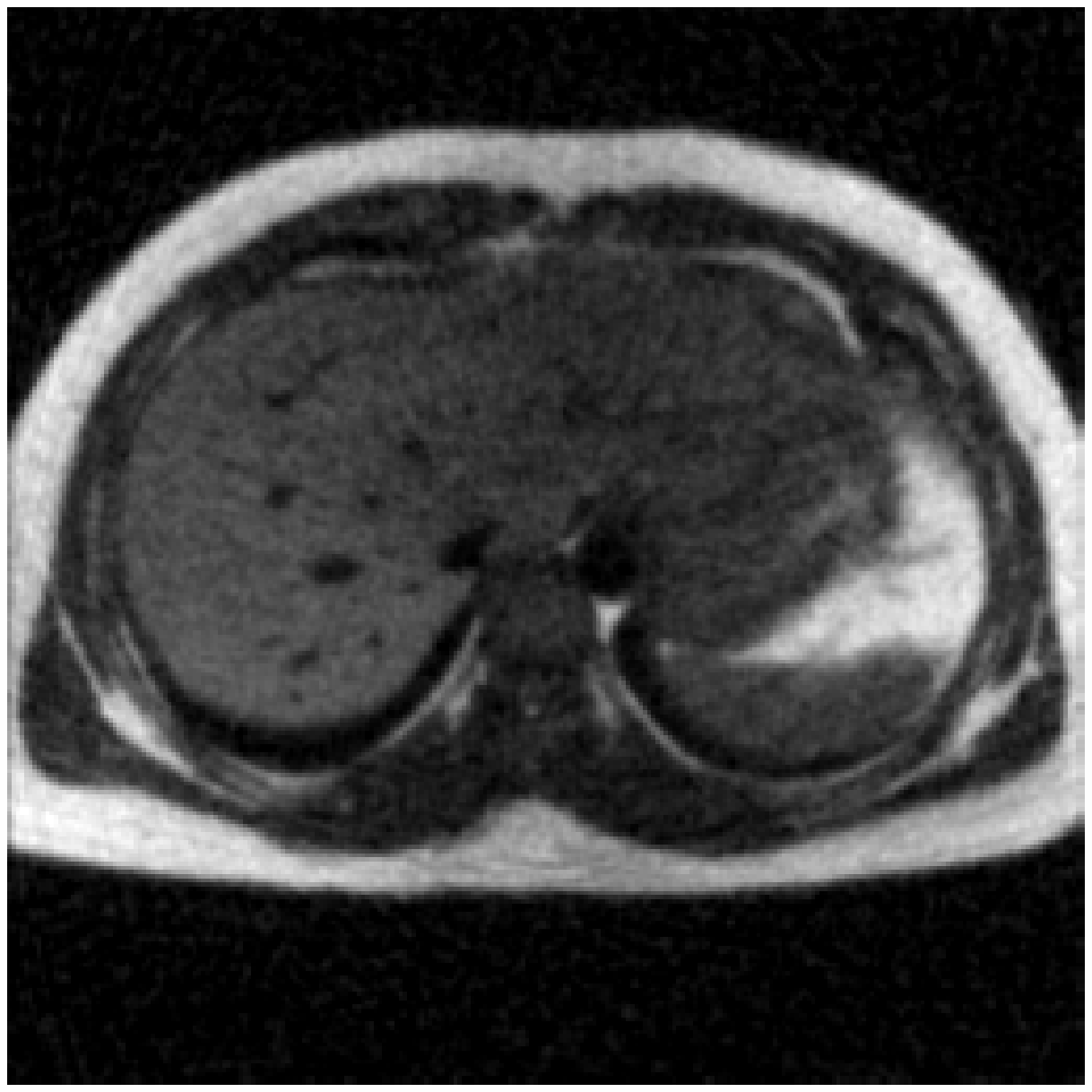} &
		\includegraphics[width=1.6in]{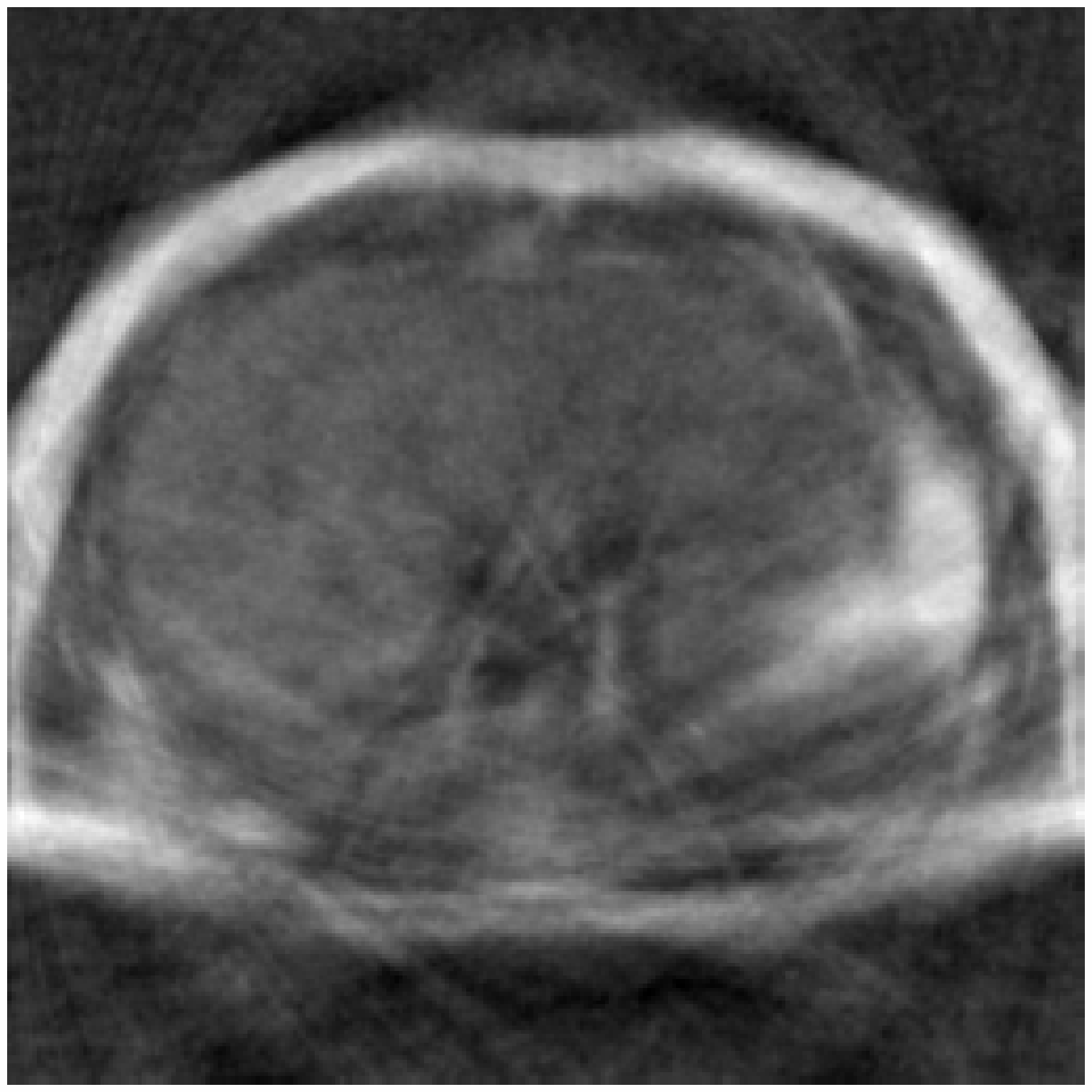} 
		\\
	(a) HyBR ($\bfgamma = \bfgamma^\true$) & (b) Motion artifacts ($\bfgamma = \bfzero$) 
	\end{tabular}
\end{center}
\caption{
The HyBR reconstruction in (a) demonstrates that a good solution can be obtained by including appropriate regularization. PAT regularized reconstruction in (b) (corresponding to the optimal regularization parameter) illustrates motion artifacts.}
	\label{fig:linearimages}
\end{figure}

However, all of these results rely on knowledge of the true motion parameters, $\bfgamma^\true$. In Figure~\ref{fig:linearimages}(b), we provide a HyBR reconstruction where $\bfA$ represents only projection (that is, motion is ignored and $\bfgamma = \bfzero$).  It is evident that, even if the optimal regularization parameter is used as is the case here, motion artifacts can severely degrade image reconstructions.


\subsection{Variable projection for simultaneous motion estimation and image reconstruction} 
\label{sub:variable_projection_for_simultaneous_estimation}
In this section, we describe a variable projection approach to solve separable nonlinear least squares problems like~\eqref{eq:sepLS} so that motion parameters can be automatically estimated and improved during the reconstruction process.  Following the discussion in Section~\ref{sub:solving_the_linear_reconstruction_problem}, since the problem is ill-posed, we consider the regularized problem
\begin{equation}
	\min_{\bff, \bfgamma} \frac{1}{2}\| \bfA(\bfgamma) \bff - \bfg \|_2^2 + \lambda^2 \norm[2]{\bff}^2\,.
\end{equation}
Although additional regularization for the motion parameters may be included, it is not needed here due to the parameterization of the motion.

The basic idea of variable projection is to mathematically eliminate the linear parameters $\bff$ and optimize over the nonlinear parameters $\bfgamma.$  For notational simplicity, for fixed $\bfgamma$, define 
\begin{equation}
	\label{eqn:linear}
	\bff(\bfgamma) = \argmin_\bff \frac{1}{2}\| \bfA(\bfgamma) \bff - \bfg \|_2^2 + \lambda^2 \norm[2]{\bff}^2,
\end{equation}
i.e., the solution to the linear Tikhonov problem.  Then the goal is to solve the \emph{reduced} nonlinear least squares problem,
\begin{equation}
	\label{eqn:reduced}
	\min_\bfgamma \frac{1}{2}\| \bfA(\bfgamma) \bff (\bfgamma) - \bfg \|_2^2\,,
\end{equation}
where any nonlinear least squares solver can be used to solve~\eqref{eqn:reduced}.  Here, we use a Gauss-Newton method since derivatives can be computed or estimated efficiently for this problem.  In particular, let $\bfr(\bfgamma) = \bfA(\bfgamma) \bff (\bfgamma) - \bfg,$ then the Jacobian with respect to $\bfgamma$ can be written as
\begin{equation}
	\label{eqn:Jacobian}
	\bfJ = \begin{bmatrix} \bfd_1 && \\ & \ddots & \\ && \bfd_n
	\end{bmatrix},\quad \mbox{where} \quad \bfd_i = \frac{\partial[\bfA_i \bfK(\gamma_i) \bff]}{\partial \gamma_i} \in \bbR^{m\times 1}\,.
\end{equation}
A Gauss-Newton approach to solve~\eqref{eqn:reduced} is provided in Algorithm~\ref{alg:GNReduced}.

\begin{algorithm}
	\caption{Gauss-Newton variable projection approach}\label{alg:GNReduced}
\begin{algorithmic}[1]
\STATE Initialize $\bfgamma^{(0)}$
\WHILE{stopping criteria not satisfied}
\STATE Solve~\eqref{eqn:linear} for $\bff(\bfgamma^{(k)})$
\STATE Compute the Jacobian matrix $\bfJ$ as~\eqref{eqn:Jacobian}
\STATE Compute step $\bfs^{(k)}$ by solving $\bfJ\t \bfJ \bfs = -\bfJ\t \bfr^{(k)}$
\STATE Update $\bfgamma^{(k+1)} = \bfgamma^{(k)} + \bfs^{(k)}$
\ENDWHILE
\end{algorithmic}
\end{algorithm}

The most computationally intensive part of the above algorithm is Step 3, solving ~\eqref{eqn:linear}.  We propose to use the hybrid LSQR approach described in \cite{ChNaOLe08} and illustrated in Section~\ref{sub:solving_the_linear_reconstruction_problem} for efficient implementation and automatic regularization parameter selection.  Also, due to our assumption that the motion represents one dimensional stretching, we can use the derivation in \cite{ChHaNa06} to get a computationally efficient analytic formula for the Jacobian.  Note that for Step 5, any linear solver could be used since the size of this problem is fairly small.  Furthermore, a line search can be included in step 6.
Additional considerations that can be incorporated include constraints on $\bfgamma$ \cite{o2013variable} and nonegativity constraints on $\bff$ \cite{cornelio2014constrained}.

\section{Numerical Results} 
\label{sec:numerical_results}
In this section, we provide numerical results demonstrating the variable projection Gauss-Newton approach for simultaneous estimation of motion parameters and PAT reconstruction. 
For this example, the goal is to reconstruct the image in Figure~\ref{fig:PATproblem}(a).  The observed sinogram in Figure~\ref{fig:PATproblem}(b) was obtained using $120$ projections from transducers located between 0 and 354 degrees at 3 degree intervals, and each projection corresponds to 363 radii. In particular, the sinogram was computed using Equation~\eqref{eqn:model} where true motion parameters $\gamma^\true \in \bbR^{120}$ were obtained using the cosine function and are provided in Figure~\ref{fig:gammacomputed}.  These motion parameters were selected to mimic rhythmic movement of the patient, e.g., due to breathing or pulsing during image acquisition.  Gaussian white noise was added to the sinogram such that the noise level was $\norm[2]{\bfe}/ \norm[2]{\bfA(\bfgamma^\true) \bff^\true} = .03.$  

Since for this problem, the size of $\bfA$ is $43,917 \times 65,536$, working with the matrix explicitly is not feasible. Most implementations avoid construction of the projection matrices $\bfA_i$ by using function calls, where each function call may require up to three \verb|for| loops (over all angles, all radii, and all intersecting pixels, as one needs to implement the numeral integrations directly).  Such approaches can be very slow, especially if many matrix-vector multiplications (mat-vecs) need to be done.  Although parallel computing could be used to accelerate the process, there is still the issue of performing unnecessary computations.  In our implementation, we exploit the fact that each projection matrix, $\bfA_i,$ is extremely sparse ($99.78\%$ of the entries are zero on average) and does not depend on $\bfgamma$.  We precompute nonzero entries for each projection image and store them in sparse matrix format.  Then we implemented matrix-vector and matrix-transpose-vector\footnote{Here, matrix-vector multiplication corresponds to motion followed by the spherical mean transform, and matrix-transpose-vector multiplication corresponds to the adjoint spherical mean transform followed by transpose with the motion matrix.} multiplications using object-oriented programming.  For this example,
evaluating a complete forward projection (over all $120$ angles and excluding motion) via function evaluation took $8.41$ seconds, whereas computing the nonzero entries for all $120$ projection matrices took $87.37$ seconds, after which each mat-vec required only $.032$ seconds\footnote{CPU timings are averaged over $100$ runs on a MacBook Pro, OSX Yosemite, 2.9 GHz Intel Core i7, 8G memory in Matlab 2015b.}.  Even with the initial cost to compute and store nonzero elements, such speed-up is significant since two mat-vecs (one matrix-vector and one matrix-transpose-vector) need to be performed at each of up to 100 iterations of the linear solver, and this is done at each Gauss-Newton iteration.
Thus, our implementation makes both iterative linear solvers and nonlinear optimization for PAT computationally feasible.

For the Gauss-Newton algorithm, we used an initial guess of no motion, i.e., $\bfgamma^{(0)}=\bfzero$.  We compare three approaches for solving the linear subproblem (line 3 in Algorithm~\ref{alg:GNReduced}): LSQR with $100$ iterations (GN-LSQR), HyBR with weighted-GCV parameter $\lambda_{\rm WGCV}$ (GN-HyBR), and HyBR with optimal regularization parameter, $\lambda_{\rm opt},$ (GN-HyBR-opt).  Results corresponding to GN-HyBR-opt are not obtainable in practice, but are provided solely as a reference.

\begin{table}
	\caption{Relative parameter error $\epsilon_\bfgamma$ and reconstruction error $\epsilon_\bff$ at each iteration of the Gauss-Newton approach, comparing three methods for solving the linear subproblem: LSQR (100 iterations), HyBR, and HyBR with the optimal regularization parameter.}
	\begin{center}
	\begin{tabular}{|c|cc||cc| c||cc|c|}\hline
& \multicolumn{2}{c||}{GN-LSQR} & \multicolumn{3}{c||}{GN-HyBR}& \multicolumn{3}{c|}{GN-HyBR-opt}\\
		\hline
GN iter & $\epsilon_\gamma$ & $\epsilon_\bff$   &  $\epsilon_\gamma$  & $\epsilon_\bff $ & $\lambda_{\rm WGCV}$   &  $\epsilon_\gamma$  & $\epsilon_\bff $ & $\lambda_{\rm opt}$  \\ \hline 
1  &  1.0000  &  0.5606 & 1.0000 &  0.4609 &    0.0268 &    1.0000 &  0.4361 & 0.0407 \\ \hline
2  &  0.8811  &  0.6297 & 0.7824 &  0.4330 &  	0.0248 &	0.6838 &  0.3851 & 0.0342 \\ \hline
3  &  0.8104  &  0.5944 & 0.6269 &  0.3976 &  	0.0238 &	0.4897 &  0.3413 & 0.0304 \\ \hline
4  &  0.7356  &  0.5737 & 0.5001 &  0.3642 &  	0.0229 &	0.3601 &  0.3071 & 0.0291 \\ \hline
5  &  0.6622  &  0.5431 & 0.3959 &  0.3350 &  	0.0221 &	0.2758 &  0.2844 & 0.0292 \\ \hline
6  &  0.5920  &  0.5187 & 0.3160 &  0.3104 &  	0.0214 &	0.2228 &  0.2712 & 0.0288 \\

\hline
\end{tabular}
\end{center}
	\label{tab:GN}
\end{table}

Relative errors at each Gauss-Newton iteration defined as
$$\epsilon_\bfgamma = \frac{\norm[2]{\bfgamma^{(k)}-\bfgamma^\true}}{\norm[2]{\bfgamma^\true}} \quad \mbox{ and } \quad \epsilon_\bff = \frac{\norm[2]{\bff(\bfgamma^{(k)})-\bff^\true}}{\norm[2]{\bff^\true}}$$
are provided in the Table~\ref{tab:GN}. We observe that in all considered cases, we were able to obtain improved estimates of the motion parameters using the nonlinear scheme.  Computed values $\bfgamma^{(6)}$ for GN-LSQR and GN-HyBR are provided in Figure~\ref{fig:gammacomputed}.  Notice that since GN-HyBR automatically incorporates regularization for solving the linear problem and hence produced more accurate reconstructions at each iteration, GN-HyBR resulted in superior estimates for the motion parameters than LSQR with 100 iterations.  Due to semiconvergence behavior (as described in Section~\ref{sub:solving_the_linear_reconstruction_problem}), standard stopping criteria for iterative methods tend to perform poorly for unregularized inverse problems.

\begin{figure}
	\includegraphics[width = 6.5in]{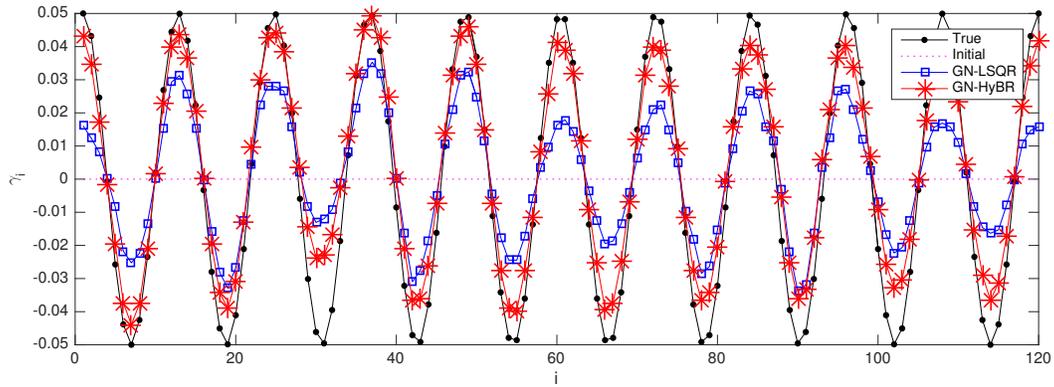}
\label{fig:gammacomputed}
\caption{Motion parameters. The black, solid line corresponds to true motion parameters $\gamma^\true$ for each transducer given by $\gamma_i = 0.05 \cos (10\, \phi_i)$.  The initial guess assumes no motion, thus $\bfgamma^{(0)} = \bfzero$ (dotted line).  Computed motion parameters after 6 Gauss-Newton iterations where LSQR and HyBR were used to solve the linear subproblem are provided in blue squares and red stars respectively.}
\end{figure}

Finally, we provide image reconstructions in Figure~\ref{fig:reconimages} corresponding to the initial Gauss-Newton reconstruction (ignoring motion), the PAT reconstruction after 6 Gauss-Newton iterations, using LSQR and HyBR to solve the linear problem.  In all of the displayed images, we restrict the pixel range to between 0 and 255 (which corresponds to the pixel range for the true image).  Thus, for visualization purposes only, any pixel value less than 0 in the displayed image gets set to 0 and any pixel larger than 255 is set to 255.

\begin{figure}
\begin{center}
	\begin{tabular}{ccc}
		\includegraphics[width = 1.5in]{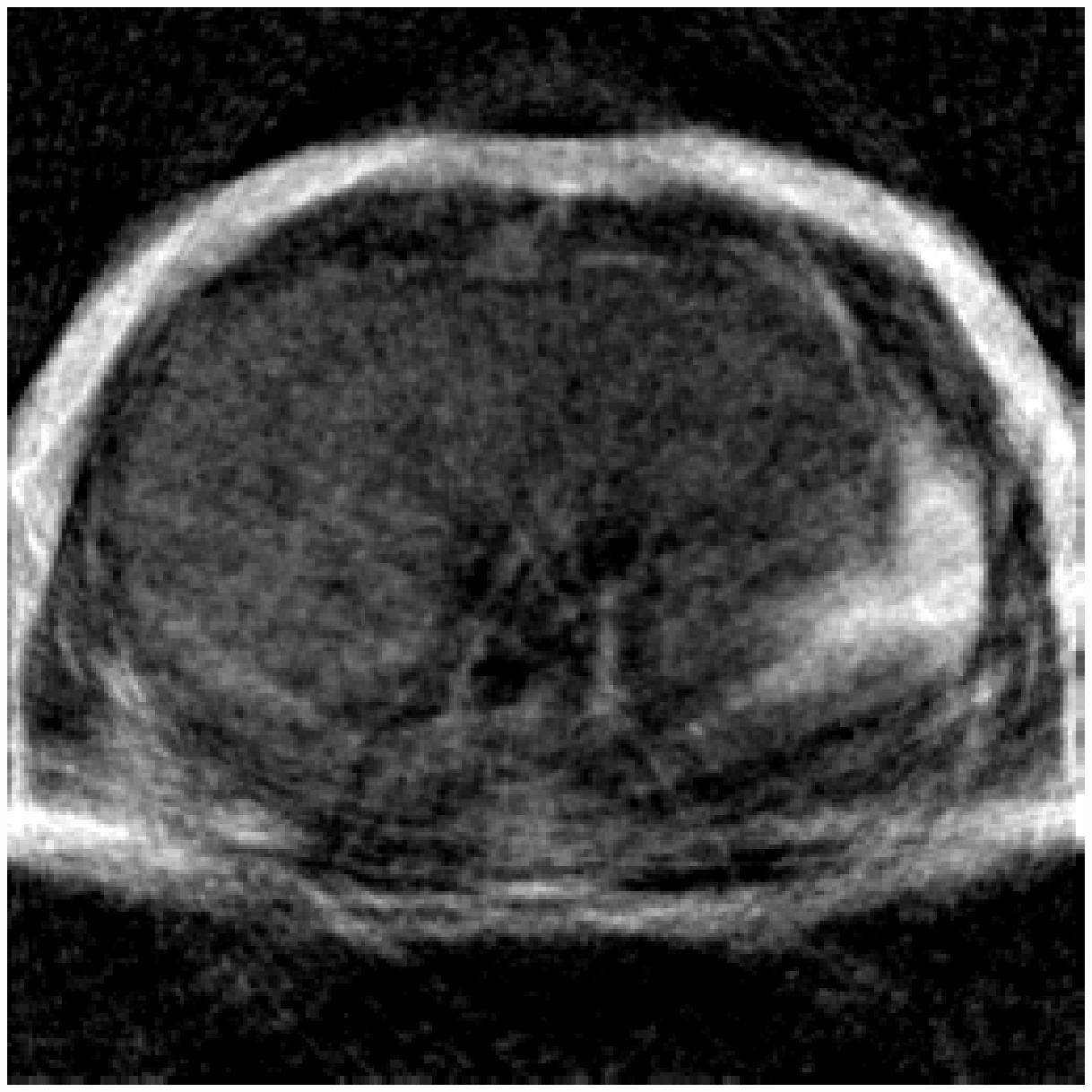} &
	\includegraphics[width = 1.5in]{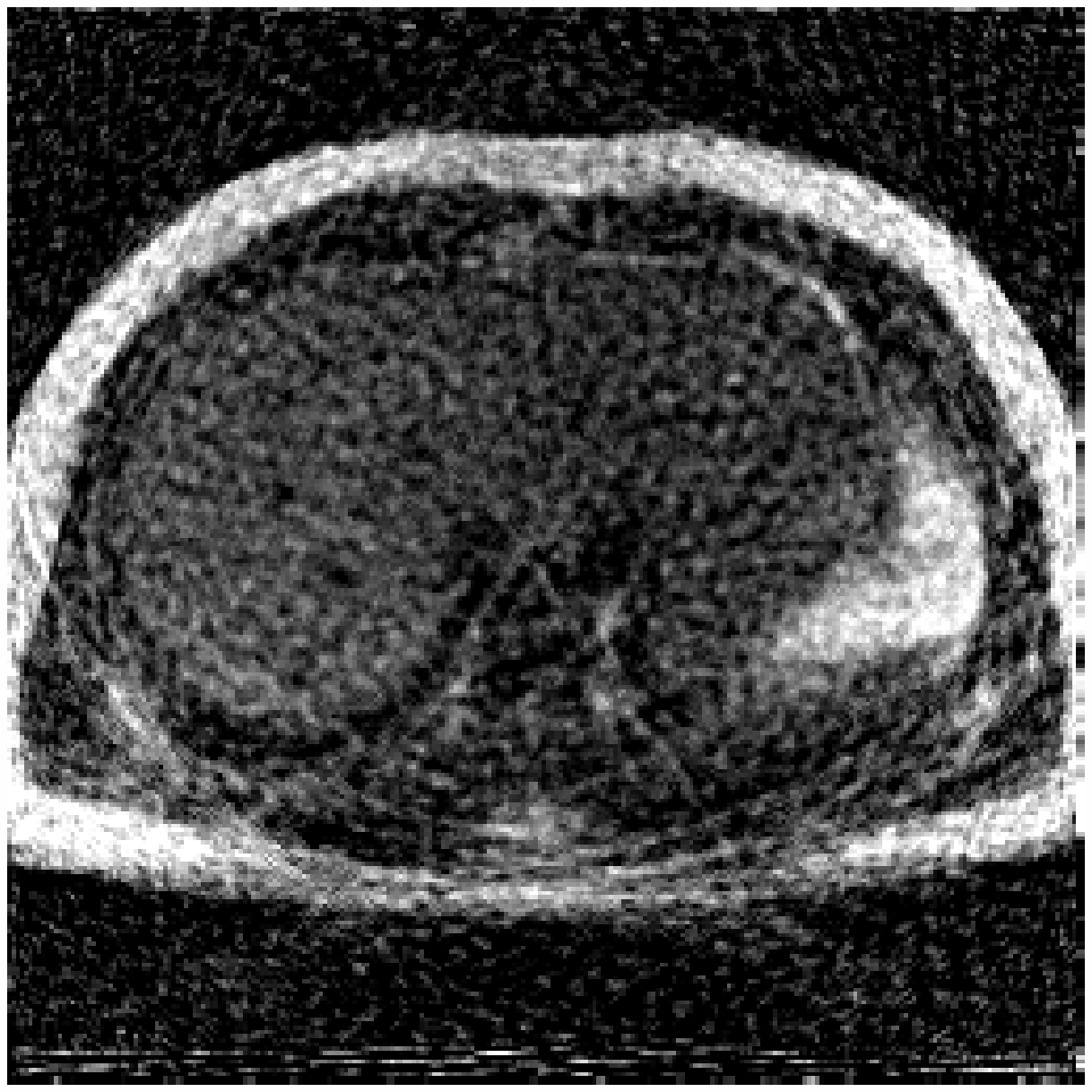} &
				\includegraphics[width = 1.5in]{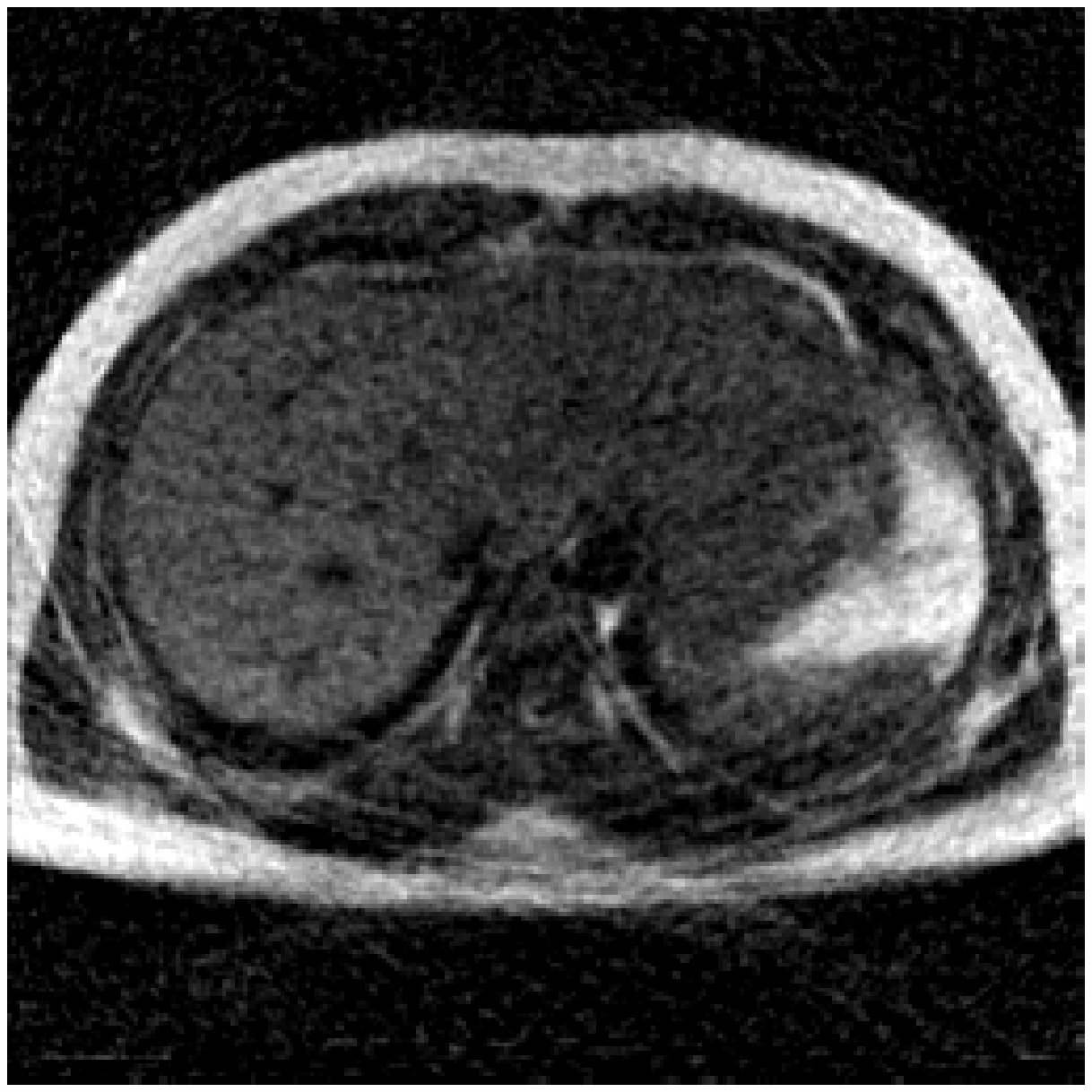}\\
				(a) GN-HyBR, $\bff(\bfgamma^{(0)})$& 			(b) GN-LSQR &  		(c)	GN-HyBR
			\end{tabular}
\end{center}
\label{fig:reconimages}
\caption{PAT image reconstructions. The regularized reconstruction in (a) is the initial reconstruction (ignoring motion) that exhibits motion artifacts.  Reconstructions in (b) and (c) were obtained after 6 Gauss-Newton iterations, where LSQR and HyBR respectively were used to solve the linear problem.}
\end{figure}

\section{Conclusions} 
\label{sec:conclusions}

In this paper, we considered the problem of motion estimation and correction in photoacoustic tomography.  We established a hybrid-type model that incorporates parameterized motion deformations in the mathematical formulation.  Then we derived two uniqueness results for the new model, under mild assumptions.  In particular, we showed that such results hold for one-dimensional stretching. Then assuming that the parameterized motion model is a good approximation to the exact motion, we formulated the problem of simultaneous motion estimation and image reconstruction as a separable nonlinear least squares problem and described a computational approach for solving the problem. We note that the described variable projection approach can be extended to other deformation models (e.g., including rotation and translation) \cite{ChHaNa06}.
The benefits of our approach are that we exploit high-level structure to handle nonlinearities, we incorporate hybrid iterative methods for automatic regularization parameter selection, and we exploit matrix sparsity and use object-oriented programming for efficiency.  Numerical results demonstrate computational feasibility and validate our approach.

\section*{Acknowledgement}
The authors thank Professor P. Kuchment, Professor T. Quinto, and Y. Lou for helpful comments. They also thank the anonymous referees their thorough reviews. L.N.'s research is partially supported by the NSF Grants DMS-1212125 and DMS-1616904. 

\appendix
\section {Proof of the equivalence of (A.5) \& (A.6) and (B) \& (C)} \label{A:Bolker}
In this section, we prove that the condition (A.5) \& (A.6) is equivalent to (B) \& (C).

Let us first prove that (A.5) is equivalent to (B). Indeed, we can easily see that $\pi_L$ is injective if and only if for all $\phi \in (\alpha,\beta)$, the mapping
$$\Pi: x \in \mathcal{O} \longmapsto (\frac{1}{2} |z(\phi) - \Phi(\phi,x)|^2, \frac{1}{2} \partial_\phi |z(\phi) - \Phi(\phi,x)|^2)$$ is injective. 
We only need to prove that $\pi_L$ is an immersion if and only if $\Pi$ is. Let us denote
$$ J = \left(\begin{array}{ccc}
    \frac{\partial^2 I}{\partial x_1 \partial \phi} & \frac{\partial^2 I}{\partial x_1 \partial r} &  \frac{\partial I}{\partial x_1}  \\[6pt] 
    \frac{\partial^2 I}{\partial x_2 \partial \phi} & \frac{\partial^2 I}{\partial x_2 \partial r} &  \frac{\partial I}{\partial x_2} \\ [6pt]
    \frac{\partial I}{\partial \phi} & \frac{\partial I}{0 \partial r} &  0\\ 
  \end{array} \right).$$
Due to \cite[Lemma 4.2]{palamodov2007remarks}, $\pi_L$ is a immersion if and only if $\det J \neq 0$.
Simple calculations show that
$$ J = \left(\begin{array}{ccc}
    \frac{\partial \Pi_2}{\partial x_1} & 0 &  \frac{\partial \Pi_1}{\partial x_1}  \\ [6 pt]
    \frac{\partial \Pi_2}{\partial x_2} & 0 &  \frac{\partial \Pi_1}{\partial x_2} \\ [6 pt]
    \frac{\partial I}{\partial \phi} & -r &  0\\ 
  \end{array} \right).$$
We, hence, obtain
$$\det J = - r \det J_\Pi,$$
where $J_\Pi$ is the Jacobian of $\Pi$. Therefore, $\det J \neq 0$ if and only if $\det J_\Pi \neq 0$. That is, $\pi_L$ is an immersion if and only if $\Pi$ is. 

Assuming that (A.5) (or, equivalently, (B)) holds, it now remains to show that (A.6) is equivalent to (C). We first note that, condition (B) also implies $\pi_R$ is a submersion (see, e.g, \cite{palamodov2007remarks}). Moreover, it is easy to check that (C) is equivalent to the fact that $\pi_R$ is surjective. Therefore, (C) is now equivalent to (A.6), i.e., $\pi_R$ is a surjective submersion. This finishes our proof.



\end{document}